\documentclass[11pt, reqno]{amsart}
\usepackage{amsmath,dsfont,amsfonts,amssymb,amsxtra,latexsym,amscd,enumerate,amsthm}
\usepackage[pagewise]{lineno}
\def\BigRoman{\uppercase\expandafter{\romannumeral\number\count 255 }}
\def\Romannumeral{\afterassignment\BigRoman\count255=}

\newcommand{\ls}{\lesssim}

\newcommand{\la}{\langle}
\newcommand{\ra}{\rangle}

\newcommand{\R}{\mathbb{R}}
\newcommand{\C}{\mathbb{C}}
\newcommand{\Z}{\mathbb{Z}}
\newcommand{\N}{\mathbb{N}}

\DeclareFontFamily{U}{mathx}{\hyphenchar\font45}
\DeclareFontShape{U}{mathx}{m}{n}{ <5> <6> <7> <8> <9> <10> <10.95>
  <12> <14.4> <17.28> <20.74> <24.88> mathx10 }{}
\DeclareSymbolFont{mathx}{U}{mathx}{m}{n}
\DeclareMathAccent{\widecheck}{0}{mathx}{"71}

\newcommand{\eps}{\epsilon} 

\DeclareMathOperator{\Real}{Re}
\DeclareMathOperator{\Imag}{Im}

\newtheorem{thm}{Theorem} \newtheorem{cor}[thm]{Corollary}
\newtheorem{pro}[thm]{Proposition} \newtheorem{lem}[thm]{Lemma}

\theoremstyle{remark} \newtheorem{rmk}[thm]{Remark}
\theoremstyle{definition} \newtheorem{dfn}[thm]{Definition}

\numberwithin{equation}{section} \numberwithin{thm}{section}

\begin{document}
\title[Scattering results for Dirac Hartree-type
equations]{Scattering results for
  Dirac Hartree-type equations \\ with small initial data}

\author[C.~Yang]{Changhun Yang} \address[C.~Yang]{Department of
  Mathematical Sciences, Seoul National University, Seoul 151-747,
  Republic of Korea} \email{maticionych@snu.ac.kr}

\begin{abstract}
We consider the Dirac equation with cubic Hartree-type nonlinearity derived by uncoupling the Dirac-Klein-Gordon systems. We prove small data scattering result in full subcritical range. Main ingredients of the proof are the localized Strichartz estimates, improved bilinear estimates thanks to null-structure hidden in Dirac operator and $Up,Vp$ function spaces. We apply the projection operator and get a system which of linear part is the Klein-Gordon type. It enables us to exploit the null-structures in equation. This result is shown to be almost optimal by showing that iteration method based on Duhamel's formula over supercritical range fails. 
\end{abstract}
\subjclass[2010]{Primary:35Q55 ; Secondary:35Q40 } \thanks{ }
\maketitle

\section{Introduction}\label{sect:intro}
We consider the following Hartree type Dirac equation
\begin{equation}\label{maineqn}
\begin{split}
(-i\partial_t+\alpha\cdot D+m\beta)\psi&=\lambda (V*\la\psi,\beta\psi\ra_{\C^4})\beta\psi, \\
\psi(0,\cdot)&=\psi_0\in H^{s}(\R^3) ( \text{ or }  \dot H^s(\R^3)),
\end{split}
\end{equation}
where $D=-i\nabla, \psi:\R^3\rightarrow \C^4$ is the Dirac spinor regarded as a column vector and $\lambda\in \R$.  $\beta$ and $\mathbf{\alpha}=(\alpha_1,\alpha_2,\alpha_3)$ are the $4\times4$ Dirac matrices given by \begin{equation}\label{matrix}
\beta=\begin{pmatrix}
I_2 & 0 \\
0 & -I_2
\end{pmatrix} , \qquad 
\alpha^j=\begin{pmatrix}
0 & \sigma^j \\
\sigma^j & 0 
\end{pmatrix}
\end{equation}
where for $j=1,2,3$ the Pauli matrices $\sigma^j$ are 
\begin{equation}
\sigma_1=\begin{pmatrix} 0&1\\1&0 \end{pmatrix}, \\
\sigma_2=\begin{pmatrix} 0&-i\\i&0 \end{pmatrix}, \\
\sigma_3=\begin{pmatrix} 1&0\\0&-1 \end{pmatrix}.
\end{equation}
The constant $m\ge0$ is a physical mass parameter and the symbol $*$ denotes convolution in $\R^3$. In this paper, we consider a generalized potential $V$ which is defined as follows:
\begin{dfn}
	For $0\le \gamma \le 2$ the potential $V$ satisfies the growth condition such that $\widehat{V} \in C^{4}(\mathbb R^3 \setminus\{0\})$ and  for $ 0 \le k \le 4$ \begin{align}\label{growth}
	|{\nabla^k}\widehat V(\xi)| \lesssim |\xi|^{-\gamma-k}\;\;\mbox{for}\;\; |\xi| \le 1, 
	\quad |{\nabla^k}\widehat V(\xi)| \lesssim |\xi|^{-2-k}\;\;\mbox{for}\;\;|\xi| > 1.
	\end{align}
\end{dfn}
The Yukawa potential $V(x)=e^{-\mu_0 |x|} |x|^{-1}, \ \mu_0>0$ is  corresponding to $\gamma=0$.
And the Coulomb potential $V(x)=|x|^{-1}$ is of such type corresponding to $\gamma=2$.
The equation \eqref{maineqn} for $\gamma=0,2$ is derived by uncoupling the following Dirac-Klein-Gordon system. 
\begin{equation}\label{eqn:DKG}
\begin{cases}
(-i\partial_t+\alpha\cdot D+ m\beta)\psi=\phi\beta\psi\\
(\partial_t^2-\Delta+M^2)\phi=\la\psi,\beta\psi\ra_{\C^4}.
\end{cases}
\end{equation}
Suppose that a scalar field $\phi$ is standing wave, i.e. $\phi(t,x)=e^{i\lambda t}\rho(x)$. Then Klein-Gordon part of \eqref{eqn:DKG} becomes
\begin{equation*}
(-\Delta-\lambda^2+M^2)\phi=\la\psi,\beta\psi\ra_{\C^4}.
\end{equation*}
Solving this equation gives
\begin{equation*}
\phi= \begin{cases}
c_1\frac{1}{4\pi|x|}*\la\psi,\beta\psi\ra_{\C^4} , \ &\text{if} \ \lambda=\pm M, \\
c_{2} \frac{e^{-\sqrt{M^2-\lambda^2}}}{|x|} *\la\psi,\beta\psi\ra_{\C^4} , \ &\text{if} \ M >|\lambda|,
\end{cases}
\end{equation*}
for some  constant $c_1$ and $c_2$. 
Putting this into the Dirac part of \eqref{eqn:DKG}, a spinor $\psi$ result in \eqref{maineqn} with potential $V$ for $\gamma=2$ and $\gamma=0$ respectively. 
We generalize these two potentials obtained from the concrete physical model into $V$ in mathematical view point.

We investigate the global behavior of solution to \eqref{maineqn}, especially scattering problem when the initial data is sufficiently small.
The scaling argument for the massless$(m=0)$ Dirac equation with Coulomb potential gives if $\psi$ is a solution of \eqref{maineqn}, then so is $\psi_{a}(t,x)=a^{\frac{3}{2}}\psi(at,ax)$, hence the scale invariant data space is $\psi_0\in L^2(\R^3)$.
Solutions to \eqref{maineqn} also satisfy conservation of charge $\|\psi(t)\|_{L_x^2}=\|\psi(0)\|_{L_x^2}$. Thus, the equation \eqref{maineqn} is charge critical. Our goal is to show the scattering result in full subcritical range and ill-posedness in supercritical range.

The previous researches have been established with the following equation which consists of the same linear part but different nonlinear term
\begin{equation}\label{subeqn}
\begin{split}
(-i\partial_t+\alpha\cdot D+m\beta)\psi&=\lambda (\frac{e^{-\mu|x|}}{|x|}*|\psi|^2)\psi, \ \mu\ge 0 \\
\psi(0,\cdot)&=\psi_0\in H^{s}(\R^3).
\end{split}
\end{equation}
More precisely, difference from \eqref{maineqn} is the way to define conjugation in inner product:
\begin{align*}
\la \psi,\psi \ra_{\C^4}&=|\psi_1|^2+|\psi_2|^2+|\psi_3|^2+|\psi_4|^2, \\
\la \psi,\beta\psi \ra_{\C^4}&=|\psi_1|^2+|\psi_2|^2-|\psi_3|^2-|\psi_4|^2.
\end{align*}
If $V$ is Coulomb type, \eqref{subeqn} appears when Maxwell-Dirac system with zero magnetic field is uncoupled \cite{CG-76}. 
On the other hand, concerning the Yukawa type it is conjectured in \cite{CG-76} that as Maxwell-Dirac case, the equation might be derived by uncoupling Dirac-Klein-Gordon system.
Concerning the well-posedness and scattering, the first well-posedness result was obtained in \cite{DF1989} when $V$ is Yukawa type. They showed existence of weak solution of \eqref{maineqn} when mass is zero(m=0). Recently, Herr and Tesfahun \cite{HT-15} established the small data scattering result for $\psi_0\in H^s$ for $s>\frac12$ when mass is nonzero.
The same results also valid for \eqref{maineqn}, since in the above results they didn't use any structure of Dirac equation.
In the paper \cite{HT-15} it is suggested that the regularity threshold can be lowered if the null structure of Dirac equation is taken account into, which initiate our interest about this research.
One of our main tool is to apply the null structure introduced in \cite{AFS-2007} but this does work only with the help of $\beta$ in nonlinear term, that is, in the case of \eqref{maineqn}. So from now on we focus on the equation \eqref{maineqn}.
There is also many results for Dirac equation with other nonlinear term. 
Dirac equation with a nonlinear term $\big(|x|^{-\gamma}*|\psi|^{p-1} \big)\psi$ for $0<\gamma<3$ and $p\ge3$ was intensively studied in \cite{MT2009}, \cite{NT-12}.
Finally let us end this paragraph with referring \cite{BH-2015} where the small data scattering result in scaling critical spaces was shown for Dirac equation with a cubic nonlinear term $\la \psi,\beta\psi \ra_{\C^4}\beta\psi$. See also the references therein.

Let us write the solution of \eqref{maineqn} using Duhamel's formula by
\begin{equation}\label{eqn:duhamel}
\psi(t)=U_m(t,D)\psi_0+\int_{0}^{t}U_m(t-\tau,D)(V*\la \psi,\beta\psi\ra)(\tau)\beta\psi(\tau)d\tau,
\end{equation}
where the linear propagator is defined on $L^2(\R^3;\C^4)$ by
\begin{equation}\label{Def:U}
U_m(t,D)=I\cos t(m^2-\Delta)^\frac12 - (\alpha+i\beta)(m^2-\Delta)^{-\frac12}\sin t(m^2-\Delta)^\frac12.
\end{equation}
One can check that the essential parts of linear propagator $U_m(t)$ is $e^{\pm it(m^2-\Delta)^\frac12}$. So one infer that some of well-posedness and scattering results for Semi-relativistic equation with Hartree Coulomb type nonlinearity would be also valid for Dirac equation. 
\begin{equation*}
i\partial_t u=\sqrt{1-\Delta}u+\lambda (|x|^{-1}*|u|^2)u,
\end{equation*}
for $u:\R^3\rightarrow\C, \ \lambda\in\R$.
The first result on the local well-posedness was shown in \cite{HL-14} for $s>\frac14$ (and $s>0$ if the initial data is radially symmetric).  Concerning the scattering, there is a negative result by \cite{CT-2006} and a modified scattering result by \cite{P-14}. We expect that similar result might be drawn for \eqref{maineqn} with $\gamma=2$ but we did not pursue here.

Now let us state the first main theorem. We use the notation $H_m^s(\R^3)$ which means homogeneous(or Inhomogeneous) Sobolev spaces $\dot H^s(\R^s)$(or $H^s(\R^3)$) if $m=0$ (or $m>0$), respectively.
\begin{thm}[Scattering results]\label{thm:scattering}
	Let $m\ge0$ and $\gamma$ be such that $$\begin{cases} 0\le \gamma<1 , &\text{ if } m>0 \\
	0\le\gamma<2 , &\text { if } m=0.
	\end{cases}$$ 
	Then for $s>0$ there exists $\delta>0$ such that for all $\psi\in H_m^s(\R^3)$ satisfying $$\|\psi_0\|_{H_m^s(\R^3)}\le \delta,$$ there exist a global solution $\psi\in C(\R,H_m^s(\R^3))$ to the Dirac equation \eqref{maineqn}. Furthermore the solutions scatter in $H_m^s(\R^3)$ in the sense that there exist $\phi\in H_m^s(\R^3)$ such that 
	\begin{align*}
	\|\psi(t)-U_m(t,D)\phi\|_{H_m^s(\R^3)} \rightarrow 0
	\end{align*}
	as time $t$ goes to infinity.
\end{thm}

The proof of main theorem is based on perturbation method, that is, we employ the standard contraction mapping argument and show that the effect of nonlinear term is negligible if an initial data is sufficiently small. 
Main tasks are to find a appropriate function spaces and estimate the nonlinear terms in it to occur contraction using Duhamel's formula. We construct a function spaces by employing the $U^2$ spaces with the dyadic decomposition. And we prove localized linear and bilinear estimates as in \cite{HT-15}.  
To overcome bad behavior from the potential near zero, we especially use the localized bilinear estimates is effectively. To utilize bilinear estimates we adapt the theory of $U^2$ spaces. 
In contrast to the previous results, by making use of null structure we enable to obtain an improved bilinear estimates and prove the scattering result in full subcritical range.

When we estimate the Hartree type nonlinearity term using dyadic decomposition, the most difficult part is to bound the high-high-low interaction because of the singularity from the potential near zero. 
The larger $\gamma$, more worse the singularity near zero in frequency side. 
So our result cannot be obtained in the full range of $\gamma$. 
For massless case $m=0$, we cannot cover the Coulomb potential $\gamma=2$. 
For massive case $m>0$, the null structure is not applicable in the low frequency part and the Strichartz estimates localized to this region is also not so satisfactory to get over the singularity, which makes us to get more restricted results than massless case.  For $1\le\gamma<2$ we might overcome this difficulty and obtain the scattering result if we choose an initial data in radial symmetry spaces or weighted function spaces. We will consider it in future work.

In supercritical range $s<0$ we will provide the ill-posedness result. Thus our result is almost sharp in respect of regularity. We give an additional assumption on the potential $V$ that $\widehat{V}$ is positive. It seems reasonable, since the Coulomb and Yukawa still are under the consideration.
\begin{thm}
Let $s<0$ and $T>0$. We further assume that $\widehat V$ is positive. Then the flow map $\psi_0\mapsto \psi$ from $H^s(\R^3)$ to $C([0,T];H^s(\R^3))$ cannot be $C^3$ at the origin.
\end{thm}

The critical case $s=0$ remains open. In future work, we will attack the critical case with some additional angular regularity assumption.

\begin{rmk}
	Wile writing on this paper, the author have learned that similar scattering result for \eqref{maineqn} with $\gamma=0$ was independently proved by Tesfahun \cite{T-2018}. 
\end{rmk}

\subsection{Notation}\label{subsec:not}
Let $\rho\in C_c^\infty(-2,2)$ be real-valued and even function satisfying $\rho(s)=1$ for
$|s|\leq 1$. For $\varphi(\xi):=\rho(|\xi|)-\rho(2|\xi|)$ define
$\varphi_k=\varphi(2^{-k}\xi)$. Then,
$\sum_{k \in \Z}\varphi_k =1$ on $\R^3\setminus\{0\}$ and it
is locally finite.  We define the (spatial) Fourier localization
operator $P_k f=\mathcal{F}^{-1}(\varphi_k \mathcal{F}f)$.
Further, we define
$\varphi_{\leq k}=\sum_{k' \in \Z: k' \le k  }\varphi_{k'}$
and
$P_{\leq k} f=\mathcal{F}^{-1}(\varphi_{\leq k}
\mathcal{F}f)$,
$P_{>k} f=f-P_{\leq k}f $.  Let
$\widetilde{\varphi}_{k}=\varphi_{k-1}+\varphi_{k}+\varphi_{k+1}$
and
$\widetilde{P}_k f=\mathcal{F}^{-1}(\widetilde{\varphi}_k
\mathcal{F}f)$.  Then $\widetilde{P}_k P_k=P_k\widetilde{P}_k=P_k$.

Next we define the Fourier localization operator $P_k^m$ depending on mass $m\ge0$ because our linear propagator shows different behaviors in the low frequency part whether the mass $m$ is zero or not. When we consider the case $m=0$, $P_k^m f=P_kf$ for all $k\in\Z$. On the other hand, if $m>0$ the operator depends on the range of $k$;
\begin{align*}
P_k^m f=\begin{cases}
0 &\text{ if } k<0, \\
P_{\le 0}f  &\text{ if } k=0, \\
P_kf  &\text{ if } k>0.
\end{cases}
\end{align*}

Now, we further make a decomposition involving the angular variable as described in \cite[Chapter$\Romannumeral 9$, Section4 ]{Steinbook}. For each $l\in \N$ we consider a equally spaced set of points with grid length $2^{-l}$ on the unit sphere $\mathbb{S}^2$, that is we fix a collection $\Omega_l:=\{ \xi_l^\nu \}_\nu$ of unit vectors that satisfy $|\xi_l^\nu-\xi_{l}^{\nu'}|\ge 2^{-l}$ if $\nu\neq\nu'$ and for each $\xi\in \mathbb{S}^2$ there exists a $\xi_l^\mu$ such that $|\xi-\xi_l^\nu|<2^{-l}$. 
Let $\mathcal{K}_l^\nu$ denote the corresponding cone in the $\xi$-space whose central direction is $\xi_j^\nu$, i.e., $\mathcal{K}_l^\nu=\{ \xi : | \frac{\xi}{|\xi|}-\xi_l^\nu | \le 2\cdot 2^{-l} \}$.
Define $\rho_l^\nu=\rho(2^{l}(\frac{\xi}{|\xi|}-\xi_l^\nu))$ and $\kappa_l^\nu=\rho_l^\nu\cdot(\sum_l \rho_l^\nu)^{-1}$. Then $\kappa_l$ is a smooth partition of unity subordinate to the covering of $\R^3\ \{0\}$ with the cone $\mathcal{K}_l^\nu$ such that each $\kappa_l^\nu$ is supported in $2\mathcal{K}_l^\nu$ and is homogeneous of degree $0$ and satisfies  
\begin{equation*}
\sum_{\nu\in\Omega_l}\kappa_l^\nu=1,  \quad
|\partial_\xi^\alpha\kappa_l^\nu(\xi)|\le A_\alpha 2^{|\alpha|l}|\xi|^{-|\alpha|}, \ \text{for all} \ \xi\neq0.
\end{equation*} 
Let $\widetilde{\kappa}_l^\nu$ with similar properties but slightly bigger support such that $\kappa_l^\nu\widetilde{\kappa_l^\nu}=1$. We define $K_l^\nu f:= \mathcal{F}^{-1}(\kappa_l^\nu \mathcal{F}f)$ and $\widetilde{K_l^\nu} f:= \mathcal{F}^{-1}(\widetilde{\kappa_l^\nu} \mathcal{F}f)$.
Then we have $I=\sum_{\nu}K_l^\nu$ and $K_l^\nu=\widetilde{K_l^\nu}K_l^\nu=K_l^\nu\widetilde{K_l^\nu}$.


\section{Null structures}

To exploit the null structure effectively we need to diagonalize the equation \eqref{maineqn}. Following \cite{AFS-2007},\cite{BH-2015} let us introduce the projection operators $\Pi_{\pm}^m(D)$ with symbol 
\begin{equation*}
\Pi_{\pm}^m(\xi)=\frac12[I\pm\frac{1}{\langle\xi\rangle}_m(\xi\cdot\alpha+m\beta)],  
\text{ where } 
\la \xi \ra_m =\begin{cases}  (m^2+|\xi|^2)^\frac12  &\text{ if}\  m>0 \\
|\xi| &\text{ if} \ m=0.
\end{cases}
\end{equation*}
Then we have the identity 
$$ \alpha\cdot D+m\beta=\la D\ra_m(\Pi_{+}^m(D)-\Pi_{-}^m(D)). $$
From this we have 
\begin{equation*}
U_m(t,D)=e^{-it\la D\ra_m} \Pi_{+}^m(D) - e^{it\la D\ra_m}\Pi_{-}^m(D),
\end{equation*}
which can be easily shown since the projection operators are idempotent, i.e., 
$$ \Pi_{\pm}^m(D)\Pi_{\pm}^m(D)=\Pi_{\pm}^m(D), \text{ and } \Pi_{\pm}^m(D)\Pi_{\mp}^m(D)=0.  $$
We denote briefly $\psi_{\pm}:=\Pi_{\pm}^m(D)\psi$ and split $\psi=\psi_++\psi_-$. By applying the operators $\Pi_{\pm}^m(D)$ to the equation \eqref{maineqn} we obtain the following system of equations
\begin{equation}\label{system}
\begin{cases}
(-i\partial_t+\langle D\rangle_m)\psi_{+}&=\Pi_{+}^m(D)[(V*\la\psi,\beta\psi\ra_{\C^4})\beta\psi], \ \psi_0^{+}\in H^s(\R^3) \\
(-i\partial_t-\langle D\rangle_m)\psi_{-}&=\Pi_{-}^m(D)[(V*\la\psi,\beta\psi\ra_{\C^4})\beta\psi], \ \psi_0^{-}\in H^s(\R^3)
\end{cases}
\end{equation}
where $\psi_0^{\pm}=\Pi_{\pm}^m(D)\psi_0$. Let us write the free solutions to \eqref{system} by $e^{-it\la D\ra_m}\psi_0^+$ and $e^{it\la D\ra_m}\psi_0^-$ respectively.

We investigate the nonlinear term with projection operator $\Pi_{\pm}^m(D)$. We decompose $\la\psi,\beta\psi\ra$ as 
\begin{align*}
&\la\psi,\beta\psi\ra = \la \Pi_{+}^m(D) \psi, \beta \Pi_{+}^m(D) \psi\ra
+\la \Pi_{+}^m(D) \psi, \beta \Pi_{-}^m(D) \psi\ra \\
&\qquad\qquad\   +\la \Pi_{-}^m(D) \psi, \beta \Pi_{+}^m(D) \psi\ra
+\la \Pi_{-}^m(D) \psi, \beta \Pi_{-}^m(D) \psi\ra.
\end{align*}
Note that $ \overline{\alpha^j}^T=\alpha^j $, which implies $\overline {  \Pi_{\theta_2}^m(\xi) }^{T}=\Pi_{\theta_2}^m(\xi)$. 
And  we can change the order of $\beta$ and $\Pi(D)$ as follows:
\begin{equation*}
\beta\Pi_{\pm}^m(D)=\Big( \Pi_{\mp}^m(D)\pm \frac{m\beta}{\langle D\rangle_m} \Big)\beta.
\end{equation*}
Using this we compute
\begin{align*}
\mathcal F \la \Pi_{\theta_1}^m(D)\psi_1,\beta\Pi_{\theta_2}^m(D)\psi_2\ra_{\C^4} (\xi)
&= \mathcal F \Big\la \Pi_{\theta_1}^m(D)\psi_1,\Big( \Pi_{-\theta_2}^m(D)\pm \frac{m\beta}{\langle D\rangle_m} \Big)\beta\psi_2\Big\ra_{\C^4} (\xi) \\
&=\int \la \Pi_{-\theta_2}^m(\eta-\xi)\Pi_{\theta_1}^m(\eta) \widehat\psi_1(\eta),
 \beta\widehat{\psi_{2}}(\eta-\xi) \ra_{\C^4} d\eta \\
&\qquad +\mathcal{F}\la  \Pi_{\theta_1}^m(D)\psi_1,\frac{m}{\la D\ra_m}\psi_2\ra_{\C^4} (\xi)
\end{align*}

We analyze the symbols from the bilinear operator. For more explanation about role of null structure in bilinear form see \cite{BH-2015}.
We first introduce the relation from \cite[Lemma2.1]{BH-2015}.
\begin{lem}\label{lem:PiPi}
Let $m\ge0$. The following holds true:
\begin{equation*}
\begin{split}
\Pi_{\pm}^m(\xi)\Pi_{\mp}^m(\eta)&= \mathcal{O}(\angle(\xi,\eta)) +\mathcal{O}(\la\xi\ra_m^{-1}+\la\eta\ra_m^{-1}) \\
\Pi_{\pm}^m(\xi)\Pi_{\pm}^m(\eta)&=\mathcal{O}(\angle(-\xi,\eta)) +\mathcal{O}(\la\xi\ra_m^{-1}+\la\eta\ra_m^{-1})
\end{split}
\end{equation*}
\end{lem}

From this relation we obtain the upperbound when the angular localization is concerned.
\begin{lem}\label{Null}
Let $m\ge0$ and $\theta_1,\theta_2\in \{+,-\}$. 
Suppose $\begin{cases}
k_1,k_2\in \N_0, 1\le l\le\min(k_1,k_2)+10  &\text{ if } m>0 \\
k_1,k_2,l\in \Z,  l\le\min(k_1,k_2)+10  &\text{ if } m=0.
\end{cases}$

For $\xi_j^{\nu},\xi_j^{\nu'}\in\Omega_j$ with $|\theta_1\xi_l^\nu-\theta_2\xi_l^{\nu'}|\le 2^{-l}$, $v,w\in \C^{4}$, we have
	\begin{equation}\label{keyestimate}
	| \la\Pi_{\theta_1}^m(2^{k_1}\xi_l^\nu )v, \beta \Pi_{\theta_2}^m(2^{k_2}\xi_l^{\nu'})w \ra_{\C^4}|
	\ls 2^{-l} |v||w|.
	\end{equation}
\end{lem}
\begin{proof}
See \cite[Lemma~3.3]{BH-2017}.
\end{proof}


\section{Linear and Bilinear Estimates}
\subsection{Free solutions}
In this subsection we prove a localized linear and bilinear estimates for free solutions of Klein-Gordon equation $e^{\pm it\la D\ra_m}f$ with $f:\R^3\rightarrow \C^4$.
First, let us introduce the Stirchartz estimates.
\begin{lem}[Klein-Gordon Strichartz estimates]\label{lem:stri}
Let $m>0$. Suppose $2\le p,q\le \infty$, $\frac2p+\frac{3}{q}=\frac32$ and $(p,q)\neq(2,\infty)$. Then
\begin{equation}\label{ineq:stri}
\|e^{\pm it\la D\ra_m}P_k^m f\|_{L_t^p L_x^q(\R^{1+3})} \ls \la 2^{k} \ra_m ^{\frac54(\frac12-\frac1q)}\|P_k f\|_{L^2(\R^3)}
\end{equation}	
\end{lem}
\begin{proof}
See \cite{COetal-11}.
\end{proof}

Next we introduce the more localized linear estimates. For detailed explanation and proof, See \cite[Lemma3.1]{BH-2017} and reference therein. We obey the notation in \cite{BH-2017}. For the convenience of reader we explicitly organize here.
For $k\in \Z$ let us consider the lattice point $L_k=2^{k}\Z^3$. Let $\eta:\R\rightarrow [0,1]$ be an even smooth function supported in the interval $[-\frac23,\frac23]$ with the property that 
$\sum_{n\in \Z} \eta(\xi-n)=1$ for $\xi\in\R$. Let $\gamma:\R^3\rightarrow[0,1], \gamma(\xi)=\eta(\xi_1)\eta(\xi_2)\eta(\xi_3)$, where $\xi=(\xi_1,\xi_2,\xi_3)$.
For $k\in\Z$ and $n\in L_k$ let $\gamma_{k,n}(\xi)=\gamma(\frac{(\xi-n)}{2^k})$. Clearly, $\sum_{n\in L_k}\gamma_{k,n}\equiv1$ on $\R^3$. We define the projection operator $\Gamma_{k,n}$ by $\Gamma_{k,n}f=\mathcal{F}^{-1}(\gamma_{k,n}\mathcal{F}f)$. Then $I=\sum_{n\in L_k}\Gamma_{k,n}$. Finally define $\widetilde{\Gamma_{k,n}}$ as before so that  $ \Gamma_{k,n}\widetilde{\Gamma_{k,n}}=\widetilde{\Gamma_{k,n}}\Gamma_{k,n}=1$. Now we are ready to give a statement.

\begin{lem}[Localized Strichartz estimates]\label{lem:localstri}
Let $m\ge0$. Suppose $\frac1p+\frac1q=\frac12$ with $2<p<\infty$. Then
\begin{equation}\label{ineq:localstri}
 \sup_{k'\le k \in \Z} 2^{-\frac kp}2^{-\frac{k'}{p}}
 \Big(\sum_{n\in \Xi_{k'}}\|\Gamma_{k',n}P_{k}^m e^{\pm it\la D\ra_m} f\|_{L_t^p L_x^q}^2\Big)^\frac12
 \ls \|f\|_{L^2(\R^3)}.
\end{equation}
\end{lem}
\begin{proof}
By orthogonality it suffices to show that 
\begin{equation}\label{ineq:freesolution}
\|\Gamma_{k',n}P_{k}^m e^{\pm it\la D\ra_m} f\|_{L_t^p L_x^q}
\ls 2^{\frac kp}2^{\frac{k'}{p}} \|f\|_{L^2(\R^3)},
\end{equation}
uniformly in $n\in L_{k'}$.	
	
We first consider the case $m>0$ and $k=0$, which reduces to show that for $k'\le 0$
\begin{equation*}
\|\Gamma_{k',n}P_{\le 0} e^{\pm it\la D\ra_m} f\|_{L_t^p L_x^q}
\ls 2^{\frac{k'}{p}} \|f\|_{L^2(\R^3)}.
\end{equation*}
This case follows from Bernstein's inequality and Strichartz estimates \eqref{ineq:stri}
\begin{align*}
\|\Gamma_{k',n}P_{\le0} e^{\pm it\la D\ra_m} f\|_{L_t^p L_x^q}
&\ls 2^{\frac{k'}{p}} \|\Gamma_{k',n}P_{\le0} e^{\pm it\la D\ra_m} f\|_{L_t^p L_x^r}, \ \frac1r=\frac12-\frac{2}{3p} \\
&\ls 2^{\frac{k'}{p}}\| P_{\le 0} e^{\pm it\la D\ra_m} f\|_{L_t^p L_x^r}\ls 2^{\frac{k'}{p}}\|  f\|_{L_x^2}.
\end{align*}
Note that we don't consider the case $m>0$ and $k<0$.
	
For the remaining case we follow the main stream of the proof in \cite[Lemma3.1]{BH-2017}.
Let $T_t$ be the operator on $L^2(\R^3)$ into $L_t^pL_x^q(\R^{1+3})$ defined by $T_t=\Gamma_{k',n}P_{k}^me^{\pm it\la D\ra_m}$. 
The $T_tT_t^*$ is a space-time convolution operator with the kernel
\begin{equation*}
K_{k',k;n}(t,x)=\int_{\R^3} e^{\pm it\la\xi\ra_m+ix\cdot\xi}\rho_k^2(\xi)\gamma_{k',n}^2(\xi)d\xi.
\end{equation*}
We claim that the kernel $K_{k',k;n}$ satisfies the following estimates
\begin{equation}\label{BoundofK}
 |K_{k',k;n}(t,x)|\ls 2^{k'} 2^k|t|^{-1},
\end{equation}
uniformly in $x$ and  $n\in L_k$.
Suppose for a moment \eqref{BoundofK} holds.
By the standard $TT^*$ argument to prove \eqref{ineq:freesolution} is equivalent to showing
\begin{equation}\label{BoundTT}
\|T_tT_t^{*}\|_{L_t^{p'}L_x^{q'}\rightarrow L_t^p L_x^q}\ls 2^{\frac{2k'}{p}}2^{\frac2pk}
\end{equation} where $p'$ is h\"older conjugate of $p$.
By Young's inequality and Plancherel's theorem we have
\begin{align*}
\|K_{k',k;n}(t,\cdot)*F(t)\|_{L_x^\infty} &\ls \| K_{k',k;n}(t,x)\|_{L_x^\infty} \|F(s)\|_{L_x^1} \\
\|K_{k',k;n}(t,\cdot)*F(t)\|_{L_x^2} &\ls \| \mathcal{F}_x K_{k',k;n}(t,\cdot)\|_{L_\xi^\infty} \|F(s)\|_{L_x^2} \ls \|F(s)\|_{L_x^2}.
\end{align*}
And by interpolation of these two estimates we obtain for $q\ge2$
\begin{equation*}
\|K_{k',k;n}(t,\cdot)*F(t)\|_{L_x^q} \ls \| K_{k',k;n}(t,x)\|_{L_x^\infty}^{1-\frac2q} \|F(s)\|_{L_x^{q'}}.
\end{equation*}
Then we estimate
\begin{align*}
\| T_tT_t^*F\|_{L_t^pL_x^q} 
&\ls \Big\| \int_{\R} \| K_{k',k;n}(t-s,\cdot)*F(s) \|_{L_x^q} ds \Big\|_{L_t^p} \\
&\ls \|  \int_{\R} \| K_{k',k;n}(t-s,\cdot)\|_{L_x^\infty}^{1-\frac2q}\|F(s)\|_{L_x^{q'}} ds \|_{L_t^p}\\
&\ls (2^{k'}2^k)^{1-\frac2q} 
\| \int |t-s|^{-1+\frac2q} \|F(s)\|_{L_x^{q'}} ds \|_{L_t^p}.
\end{align*}
Then Hardy-Littlewood-Sobolev inequality with $0<\frac2q=\frac{1}{p'}-\frac1p<1$ implies \eqref{BoundTT}.

Finally we prove \eqref{BoundofK}. Rescaling yields
\begin{align*}
K_{k',k;n}(t,x)=2^{3k} K_{k'-k,1,2^{-k}n}(2^kt,2^kx)
\end{align*}
Then it suffices to show that for $\ |a|\sim 1$
\begin{equation}\label{ineq:K}
|K_{k'-k,1;a}(s,y)|\ls  2^{k'-k}|s|^{-1},
\end{equation} 
where for $\la\xi\ra_{k;m}=(|\xi|^2+2^{-2k}m)^\frac12$,
\begin{equation*}
K_{k'-k,1;a}(s,y)=\int_{\R^3} e^{\pm is\la \xi\ra_{k}+iy\cdot\xi}\rho_1^2(\xi)\gamma_{k'-k,a}^2(\xi)d\xi, \ .
\end{equation*}
By rotation, we may assume that $y=(0,0,|y|)$. We change of variables using spherical coordinates:
\begin{align*}
K_{k'-k,1;a}(s,y)&=\int_0^\infty\int_0^{2\pi}\int_0^{\pi}
e^{i(|y|r\cos\theta+s\la r\ra_k)}\zeta_{k',k}(\theta,\phi,r)\sin\theta r^2d\theta d\phi dr,  
\end{align*}
where $\zeta_{k',k}$ is smooth cutoff function supported in a thickened spherical cap of size $2^{k'-k}$ located near unit sphere.
The derivative of phase function with respect to $r$ is $|y|\cos\theta+s\frac{r}{\la r\ra_{k;m}}$. Thus the worst case occurs when $0<\theta\ll1$ and $|y|\sim  2^{k}\la2^{k}\ra_{k;m}^{-1}|s|$, otherwise since the derivative of phase function has a lower bound we can perform an integration by parts arbitrarily many times and get sufficient decay. So we only discuss the first case, where $\zeta_{k',k}(r)$ and $\zeta_{k',k}(\theta)$ is supported in an interval of length $\ls 2^{k'-k}$ in $(\frac{1}{4},4)$ and $[0,\pi)$ respectively and $|\partial_\theta\zeta_{k',k}|\ls 2^{k-k'}$.
We integrate by parts with respect to $\theta$:
\begin{align*}
K_{k'-k,1,a}(s,y)&=\int_0^\infty\int_0^{2\pi}\bigg[ \frac{i}{|y|}e^{i(|y|r\cos\theta+s\la r\ra_{k;m})} \zeta_{k',k}(\theta,\phi,r) \bigg]_0^\pi r d\phi dr \\
&-\frac{i}{|y|} \int_0^\infty\int_0^{2\pi}\int_0^{\pi}
e^{i(|y|r\cos\theta+s\la r\ra_{k;m})}\partial_\theta \zeta_{k',k}(\theta,\phi,r) rd\theta d\phi dr.
\end{align*}
The the support properties of $\zeta_{k',k}$ imply
\begin{equation*}
 | K_{k'-k,1,a} (s,y)|\ls 2^{k'-k}|y|^{-1},
\end{equation*}
which implies \eqref{ineq:K} since $2^{k}\la2^{k}\ra_{k;m}^{-1}\sim 1$.

\end{proof}

We prove a bilinear estimates for free solutions which propagate to the same direction.
\begin{lem}\label{Lem:bilinear for free}
	Let $m\ge0$. Suppose $\begin{cases}
	k\in\Z \text{ and } k_1,k_2\in \N,   &\text{ if } m>0 \\
	k,k_1,k_2\in \Z,  &\text{ if } m=0
	\end{cases}$ satisfy $2^k\ll 2^{k_1}\sim 2^{k_2}$.  Then we have
	\begin{align}\label{Bilinear estimate}
	\| P_k \la P_{k_1}^m e^{\pm it\la D\ra_m}f, P_{k_2}^m e^{\pm it\la D\ra_m} g\ra_{\R^4}\|_{L_{t,x}^2(\R^{1+3})}
	\ls 2^k \| f\|_{L^2(\R^3)} \| g\|_{L^2(\R^3)},
	\end{align}
where $\la f(x),g(x)\ra_{\R^4} := \sum_{j=1}^4 f_j(x) g_j(x)$.
\end{lem}
\begin{proof}
We may assume $\theta_1=\theta_2=+$. We make a localization by a cube of size $2^k$
	\begin{align*}
	&\| P_k \la P_{k_1}^me^{it\la D\ra_m}f, P_{k_2}^m e^{it\la D\ra_m} g\ra_{\R^4} \|_{L_{t,x}^2(\R^{1+3})}
	\ls \sum_{|n_i|\sim 2^{k_i},|n_1+n_2|\sim 2^k} 
	\| I_{n_1,n_2}\|_{L_{t,x}^2(\R^{1+3})} , \\
	&\text{where}\ I_{n_1,n_2}(t,x) = \la \Gamma_{k,n_1}P_{k_1}^me^{it\la D\ra}f(x), \Gamma_{k,n_2} P_{k_2}^m e^{it\la D\ra} g(x)\ra_{\R^4}.
	\end{align*}	
	We decompose the integral $I_{n_1,n_2}$ into
	\begin{align*}
	I_{n_1,n_2}(t,x)
	&=\int_{\R^3}\int_{\R^3} e^{ix\cdot(\xi+\eta)+it\sqrt{m+|\eta|^2}+it\sqrt{m+|\xi|^2}}
	\la \rho_{k_1}(\xi)\gamma_{k,n_1}(\xi)\widehat{f}(\xi),
	\rho_{k_2}(\eta)\gamma_{k,n_2}(\eta)\widehat{g}(\eta)\ra_{\R^4} d\eta d\xi \\
	&=:I_{n_1,n_2}^1+I_{n_1,n_2}^2+I_{n_1,n_2}^3, \\
	I_{n_1,n_2}^i(t,x)&=\int_{A_i}\int_{\R^3} e^{i(x,t)\cdot(\xi+\eta,\sqrt{m+|\eta|^2}+\sqrt{m+|\xi|^2})}
	\la \rho_{k_1}(\xi)\gamma_{k,n_1}(\xi)\widehat{f}(\xi),
	\rho_{k_2}(\eta)\gamma_{k,n_2}(\eta)\widehat{g}(\eta)\ra_{\R^4} d\eta d\xi.
	\end{align*}
	Here, we chosen almost disjoint sets $A_i$ for $i=1,2,3$ such that  $$\R^3=\cup_{i=1}^3A_i \text{ and } 
	A_i\subset \{ \xi=(\xi_1,\xi_2,\xi_3)\in\R^3: |\xi|\sim |\xi_i| \}.$$ 
	We make the change of variable $(\eta,\xi_i)\mapsto (\xi+\eta,\sqrt{m+|\eta|^2}+\sqrt{m+|\xi|^2})=\zeta=(\zeta_1,\cdots,\zeta_4)$ with 
	\begin{align*}
	d\eta d\xi_i= \Big|\frac{\partial(\zeta_1,\cdots,\zeta_4)}{\partial(\xi_i,\eta_1,\eta_2,\eta_3)}\Big|^{-1} d\zeta.
	\end{align*}
	We estimate the Jacobi on $A_i$	
	\begin{align}\label{bound for jacobi}
     \Big| \frac{\partial(\zeta_1,\cdots,\zeta_d)}{\partial(\xi_i,\eta_1,\eta_2,\eta_3)} \Big|
	= \Big|\frac{\eta_i}{\sqrt{m+|\eta|^2}} - \frac{\xi_i}{\sqrt{m+|\xi|^2}}\Big|
	\sim 1,
	\end{align}
	since $|n_1+n_2|\sim 2^k$.
Now we fix $i=1$. The other cases can be estimated similarly.
We have by Minkowski's inequality
\begin{align*}
\| I_{n_1,n_2}^1\|_{L_{t,x}^2(\R^{1+3})}
&\ls \int_{ \{ |(\xi_2,\xi_3)|\ls 2^k \} } 
 \Big\| \int_{\R} \int_{\R^3} e^{i(x,t)\cdot(\xi+\eta,\sqrt{m+|\eta|^2}+\sqrt{m+|\xi|^2})} \\
 &\qquad\qquad\times \la \rho_{k_1}(\xi)\gamma_{k,n_1}(\xi)\widehat{f}(\xi),
 \rho_{k_2}(\eta)\gamma_{k,n_2}(\eta)\widehat{g}(\eta)\ra_{\R^4} d\eta d\xi_1 \Big\|_{L_{t,x}^2}
d\xi_2d\xi_3. 
\end{align*}	
And by applying Plancherel's theorem with respect to variable $(x,t)$ and reversing the change of variables with \eqref{bound for jacobi} we estimate
\begin{align*}
\| I_{n_1,n_2}^1\|_{L_{t,x}^2(\R^{1+3})}
	\ls \int_{ \{ |(\xi_2,\xi_3)|\ls 2^k \} }  \Big\|  \la \rho_{k_1}(\xi)\gamma_{k,n_1}(\xi)\widehat{f}(\xi),
	\rho_{k_2}(\eta)\gamma_{k,n_2}(\eta)\widehat{g}(\eta)\ra_{\R^4} \Big\|_{L_{\eta,\xi_1}^2} d\xi_2 d\xi_3.
\end{align*}	
And by Cauchy Schwarz inequality we obtain
\begin{align*}
\| I_{n_1,n_2}^1\|_{L_{t,x}^2(\R^{1+3})}
&\ls	2^k  \| \la \rho_{k_1}(\xi)\gamma_{k,n_1}(\xi)\widehat{f}(\xi)
	\rho_{k_2}(\eta)\gamma_{k,n_1}(\eta)\widehat{g}(\eta) \ra_{\R^4} \|_{L_{\eta,\xi}^2} \\
&\ls 2^k \| P_{k_1}^m\Gamma_{k,n_1} f\|_{L^2(\R^3)} \| P_{k_2}^m\Gamma_{k,n_2} g\|_{L^2(\R^3)}.
	\end{align*}
Then we finally have by 
\begin{align*}
	\| P_k \la P_{k_1}^me^{it\la D\ra_m}f, P_{k_2}^m e^{it\la D\ra_m} g \ra_{\R^4} 
	\|_{L_{t,x}^2(\R^{1+3})}
	&\ls \sum_{|n_1+n_2|\sim 2^k}
	2^k \| P_{k_1}^m\Gamma_{k,n_1} f\|_{L^2} \| P_{k_2}\Gamma_{k,n_2} g\|_{L^2} \\
	&\ls 2^k \| P_{k_1}^mf\|_{L^2} \| P_{k_2}^mg\|_{L^2},
\end{align*}
where we used Cauchy Schwarz inequality with respect to $n$ variable.
\end{proof}

\subsection{Transference Principle}
Let $1\leq p<\infty$. We call a finite set $\{t_0,\ldots,t_K\}$ a partition if $-\infty<t_0<t_1<\ldots<t_K\leq \infty$, and denote the set of all partitions by $\mathcal{T}$. A corresponding step-function $a:\R\to L^2(\R^3)$ is called $U^p_{\pm\la D\ra_m}$-atom if
\[
a(t)=\sum_{k=1}^K \mathbf{1}_{[t_{k-1},t_k)} (t)e^{\mp it\la D\ra_m}\varphi_k, \quad \sum_{k=1}^K\|\varphi_k\|_{L^2(\R^3)}^p=1,\quad \{t_0,\ldots,t_K\}\in \mathcal{T},
\]
and $U^p_{\pm\la D\ra}$ is the atomic space. Further, let $V^p_{\pm\la D\ra_m}$ be the space of all right-continuous $v:\R \to L^2(\R^3)$ satisfying
\[
\|v\|_{V_{\pm\la D\ra_m}^p}:=\sup_{\{t_0,\ldots,t_K\}\in \mathcal{T}}\big(\sum_{k=1}^K\|e^{\pm it_k\la D\ra_m}v(t_k)-e^{\pm it_{k-1}\la D\ra_m}v(t_{k-1})\|_{L^2(\R^3)}^p\big)^{\frac1p}.
\]
with the convention $e^{\pm it_K \la D\ra_m}v(t_K)=0$ if $t_K=\infty$.
For the theory of $U^p_{\pm\la D\ra_m}$ and $V^p_{\pm\la D\ra_m}$, see e.g.\
\cite{HHetal-09,HHetal-10,KTetal-14}.

From now, we prove linear and bilinear estimates in $V_{\pm\la D\ra_m}^2$ by transferring the ones for free solutions in the previous section. 

\begin{cor} \label{Cor:localstri}
Let $m\ge0$. Let  $k',k,l \in \Z$, $k'\le k$ and  $2<p<\infty$.  Then we have for $\psi\in V_{\pm\la D\ra_m}^2$	
\begin{equation}\label{ineq:V2}
\Big(\sum_{\nu\in\Omega_l}\sum_{n\in L_{k'}}\| K_{l}^\nu \Gamma_{k',n}P^m_{k} \psi \|_{L_t^p L_x^q}^2\Big)^\frac12
\ls 2^{\frac kp}2^{\frac{k'}{p}} \|P_k^m \psi\|_{V_{\pm\la D\ra_m}^2}.
\end{equation}
\end{cor}
\begin{proof}
By the atomic structure of $U_{\pm\la D\ra_m}^p$, estimates in $L^2$ for free solutions transfer to $U_{\pm\la D\ra_m}^p$ functions, hence to $V_{\pm\la D\ra_m}^2$. Thus from  \eqref{ineq:freesolution} we have 
\begin{equation}\label{ineq:localv2}
\| \Gamma_{k',n}P_{k}^m \psi \|_{L_t^p L_x^q} 
\ls  2^{\frac kp}2^{\frac{k'}{p}}
\|  \psi \|_{U_{\pm\la D\ra_m}^p} 
\ls 2^{\frac kp}2^{\frac{k'}{p}}
\|  \psi \|_{V_{\pm\la D\ra_m}^2}.
\end{equation}
Actually we can check \eqref{ineq:localv2} also holds with operator changed into $\widetilde{K_l^\nu}\widetilde{P}_k^m$ from which we induce 
\begin{equation*}
\| K_{l}^\nu \Gamma_{k',n}P_{k}^m \psi \|_{L_t^p L_x^q} \ls 2^{\frac kp}2^{\frac{k'}{p}}
\| K_{l}^\nu \Gamma_{k',n}P_{k}^m \psi \|_{V_{\pm\la D\ra_m}^2}
\end{equation*}
Since projection with respect to $\nu$ and $n$ is almost disjoint, the definition of $V_{\pm\la D\ra}^2$ implies
\begin{align*}
\Big(\sum_{\nu}\sum_{n\in L_{k'}}\| K_{l}^\nu \Gamma_{k',n}P_{k}^m \psi \|_{L_t^p L_x^q}^2\Big)^\frac12
&\ls  2^{\frac kp}2^{\frac{k'}{p}}
\Big(\sum_{\nu}\sum_{n\in L_{k'}}\| K_{l}^\nu \Gamma_{k',n}P_{k}^m \psi \|_{ V_{\pm\la D\ra_m}^2 }^2\Big)^\frac12  \\
&\ls 2^{\frac kp}2^{\frac{k'}{p}} \| P_k^m \psi \|_{V_{\pm \la D\ra_m}^2}.
\end{align*}
\end{proof}

We obtain bilinear estimates in $V^2$ from above linear one.
\begin{cor}\label{Cor:bilinear}
	Let $m\ge0$, $\theta_1,\theta_2\in\{1,-1\}$ and $k_1,k_2\in\Z$. Then for $\psi_i\in V_{\theta_i\la D\ra}^2$ satisfying $\Pi_{\theta_i}^m(D)P_{k_{i}}^m\psi_i=\psi_i$, $i=1,2$.
	Then we have for $0<s<\frac12$
	\begin{equation}\label{bilinear:HighLow}
	\| \la \psi_{1}, \beta\psi_{2} \ra_{\C^4} \|_{L^2(\R^{1+3})}
	\ls 2^{sk_1}2^{sk_2} 2^{(1-2s)\min(k_1,k_2)}
	\|\psi_{1}\|_{V_{\theta_1 \la D\ra_m}^2}
	\|\psi_{2}\|_{V_{\theta_2 \la D\ra_m}^2}.
	\end{equation}
\end{cor}
\begin{proof}
Letting \eqref{ineq:V2} with $k'=k$ gives
\begin{align*}
\| P_k^m \psi \|_{L_t^pL_x^q} \ls 2^\frac{2k}{p} \| P_k^m\psi\|_{V_{\pm\la D\ra_m}^2},
\end{align*}
for $2<p,q<\infty$ satisfying $\frac1p+\frac1q=\frac12$.
Using this, we have 
\begin{align*}
\| \la \psi_{1}, \beta\psi_{2} \ra_{\C^4} \|_{L^2(\R^{1+3})}
\ls \|\psi_1\|_{L_t^p L_x^q} \|\psi_2\|_{L_t^qL_x^p}
\ls 2^{\frac{2}{p}k_1} 2^{(1-\frac{2}{p})k_2} \|\psi_1\|_{V_{\theta_1\la D\ra_m}^2} \|\psi_2\|_{V_{\theta_2\la D\ra_m}^2},
\end{align*}
which implies \eqref{bilinear:HighLow} if we let $\frac2p=s$.
\end{proof}

If we apply the null structure in Dirac operator induced by projection operator $\Pi_{\pm}^m(D)$ we can obtain improved one for high-high-low interaction case. 
\begin{pro} Let $m\ge0$. Suppose $\theta_1\theta_2=1$ and $\begin{cases}
	k\in\Z \text{ and } k_1,k_2,\in \N \ &\text{ if } \ m>0 \\
	k,k_1,k_2 \in \Z &\text{ if } \ m=0
	\end{cases}$
satisfy $2^k\ll 2^{k_1}\sim 2^{k_2}$.	Then
 for $\psi_i\in V_{\theta_i\la D\ra}^2$ satisfying $\Pi_{\theta_i}^m(D)P_{k_{i}}^m\psi_i=\psi_i$, $i=1,2$ we have  
	\begin{equation}\label{bilinear:hhl:high1}
	\|P_{k} \la \psi_{1}, \beta \psi_{2} \ra_{\C^4} \|_{L^2(\R^{1+3})}
	\ls 2^{\frac{3k}{2}-\frac{k_1}{2}}
	\|\psi_{1}\|_{V_{\theta_1 \la D\ra_m}^2}
	\|\psi_{2}\|_{V_{\theta_2 \la D\ra_m}^2}.
	\end{equation}
\end{pro}
\begin{proof}
We divide the left-hand side applying cube $n\in L_{k}$ and cone $\mathcal{K}_l^\nu$, $\nu\in\Omega_l$  decomposition with $2^{-l}\sim2^{k-\max(k_1,k_2)}$. Note that by the support condition in this case, the summation over cube is restricted to $|n-n'|\preceq k$ and the summation over cone to $|\theta_1\xi_{l}^\nu-\theta_2\xi_{l}^{\nu'}|=|\xi_1-\xi_2|\sim 2^{-l}$.
	\begin{align*}
	&\|P_{k} \la \psi_{1}, \beta_j \psi_{2} \ra_{\C^4} \|_{L^2(\R^{1+3})} \\
	&\ls 
	\sum_{\substack{n,n'\in L_{k}\\
			|n-n'|\preceq k}}  
	\sum_{\substack{\xi_{l}^{\nu},\xi_{l}^{\nu'}\in\Omega_j\\
			|\theta_1\xi_{l}^{\nu}-\theta_2\xi_{l}^{\nu}|\ls 2^{-l}}}
	\| P_{k} \la K_l^\nu \Gamma_{k,n} \psi_{1}, \beta_j  K_l^{\nu'}\Gamma_{k,n'}\psi_{2}\ra_{\C^4} \|_{L^2(\R^{1+3})} \\
	&\ls 
	\sum_{\substack{n,n'\in L_{k}\\
			|n-n'|\preceq k}}  
\sum_{\substack{\xi_{l}^{\nu},\xi_{l}^{\nu'}\in\Omega_j\\
		|\theta_1\xi_{l}^{\nu}-\theta_2\xi_{l}^{\nu}|\ls 2^{-l}}}
	2^{-l}\|  K_l^{\nu}\Gamma_{k,n} \psi_{1}\|_{L^4(\R^{1+3})} \|K_l^{\nu'}\Gamma_{k,n'}\psi_{2} \|_{L^4(\R^{1+3})},
	\end{align*}
	where in the last inequality null structure is exploited by Lemma~\ref{Null} since $\Pi_{\theta_i}^m(D)\psi_i=\psi_i$ and $2^{k_i}>1$.  
	Applying the Cauchy-Schwarz inequality with $n,n'$ and $\nu,\nu'$ and then using \eqref{ineq:V2} we finally get
	\begin{align*}
	\|P_{k} \la \psi_{1}, \beta_j \psi_{2} \ra_{\C^4} \|_{L^2(\R^{1+3})}\ls  2^{k-k_1} 2^\frac{k}{2}2^{\frac{k_1}{2}} \|\psi_{1,k_1}\|_{V_{\theta_i \la D\ra}^2}
	\|\psi_{2,k_2}\|_{V_{\theta_i \la D\ra}^2} .
	\end{align*}
\end{proof}

Next one follow from bilinear estimates in Lemma~\ref{Lem:bilinear for free}.
\begin{pro}\label{Prop;V2}
	Let $m\ge0$. Suppose $\theta_1\theta_2=-1$ and $\begin{cases}
	k\in\Z \text{ and } k_1,k_2\in \N,   &\text{ if } m>0 \\
	k,k_1,k_2\in \Z,  &\text{ if } m=0
	\end{cases}$ satisfy $2^k\ll 2^{k_1}\sim 2^{k_2}$.   	Then
	for $\psi_i\in V_{\theta_i\la D\ra}^2$ satisfying $\Pi_{\theta_i}^m(D)P_{k_{i}}^m\psi_i=\psi_i$, $i=1,2$ we have 
\begin{equation}\label{bilinear:hhl:high2}
\|P_{k} \la \psi_{1}, \beta \psi_{2} \ra_{\C^4} \|_{L^2(\R^{1+3})}
\ls 2^{k}
\|\psi_{1}\|_{V_{\theta_1 \la D\ra_m}^2}
\|\psi_{2}\|_{V_{\theta_2 \la D\ra_m}^2}.
\end{equation}	
\end{pro}
\begin{proof}
 We suffices to show that for $\psi_1\in V_{\theta_1\la D\ra_m}^2, \psi_2\in V_{\theta_2\la D\ra_m}^2$		
\begin{align*}
\| P_k  \la P_{k_1}^m\psi_1 , \beta P_{k_2}^m \psi_2 \ra_{\C^4}  \|_{L_{t,x}^2(\R^{1+3})}
\ls 2^k \| \psi_1\|_{V_{\theta_1\la D\ra_m}^2} \| \psi_2\|_{V_{\theta_2\la D\ra_m}^2}
\end{align*}
	Since $e^{\pm i t\la D\ra} \psi = \overline { e^{\mp i t\la D\ra} \bar{ \psi} } $ ,
	we have $\| \psi \|_{V_{\pm\la D\ra}^2} = \| \bar \psi \|_{V_{\mp\la D\ra}^2}$.
	Thus it reduces to show that
	\begin{align*}
	\| P_k  \la P_{k_1}^m\psi_1 , \beta P_{k_2}^m \psi_2 \ra_{\R^4} \|_{L_{t,x}^2(\R^{1+3})}
	\ls 2^k \| \psi_1\|_{V_{\pm\la D\ra_m}^2} \| \psi_2\|_{V_{\pm\la D\ra_m}^2},
	\end{align*}
	where $\la \psi,\phi \ra_{\R^4} = \sum_{i=1}^4 \psi_i\phi_i $.	
	This follows from a slight modification of the proof in \cite[Proposition~2.5]{HY-2017} by using \eqref{Bilinear estimate}.	
We give a explicit proof in Appendix for convenience of readers. 
\end{proof}

Collecting all the cases we prove above, we arrange them in the following Corollary.
\begin{cor}\label{Cor:bilinear2}
Let $m\ge0$, $\theta_1,\theta_2\in\{-1,1\}$. Suppose $k,k_1,k_2\in\Z$ satisfy $2^k \ls 2^{k_1}\sim 2^{k_2}$. For $\psi_i\in V_{\theta_i\la D\ra}^2$ satisfying $\Pi_{\theta_i}(D)P_{k_{i}}^m\psi_i=\psi_i$, $i=1,2$ we have
\begin{equation}\label{bilinear:hhl:high}
\|P_{k} \la \psi_{1}, \beta\psi_{2} \ra_{\C^4} \|_{L^2(\R^{1+3})}
\ls B_m(k)
\|\psi_{1}\|_{V_{\theta_1 \la D\ra}^2}
\|\psi_{2}\|_{V_{\theta_2 \la D\ra}^2},
\end{equation}
where $B_m(k)=\begin{cases}
2^\frac{k}{2} &\text{ if } m>0\ \& \  \min(k_1,k_2)=0 \\
2^k &\text{ otherwise }.
\end{cases}$
\end{cor}
\begin{proof}
For \eqref{bilinear:hhl:high} where $m>0$ and one of $k_1,k_2$ is zero, 
we apply Bernstein inequality 
\begin{align*}
\|P_{k} \la \psi_{1}, \beta_j \psi_{2} \ra_{\C^4} \|_{L^2(\R^{1+3})}
&\ls 2^{\frac{k}{2}}
\|P_{k} \la \psi_{1}, \beta_j \psi_{2} \ra_{\C^4} \|_{L_t^2L_x^{\frac32}(\R^{1+3})}\\
&\ls 2^{\frac{k}{2}}\|\psi_{1}\|_{L_t^4 L_x^3} \|\psi_{2}\|_{L_t^4 L_x^3}.
\end{align*}	
Applying \eqref{lem:stri} and the inclusion
$V_{\pm\la D\ra}^2\hookrightarrow U_{\pm\la D\ra}^p$ for $p>2$ gives the desired result.
\end{proof}


\section{Proof of Main results}

Our resolution space $X_{\pm}^{s}$ corresponding to the Sobolev spaces regularity $s$ is the space of functions in $C(\R,H^s(\R^3;\C^4))$ such that
\begin{equation*}
X_{\pm}^{s}:=\left\{
\| u\|_{X_{\pm}^{s}}:=\bigg(\sum_{k\in\Z }
 2^{2sk}   \| P_k^m  u\|_{U_{\pm\la D\ra_m}^2}^2 \bigg)^\frac12 <\infty
\right\}
\end{equation*}
For given $\psi\in H^s(\R^3)$ we decompose $\psi=\Pi_+^m(D)\psi+\Pi_-^m(D)\psi=:\psi_+ + \psi_-$ and define $X^s$ by
\begin{equation*}
\psi\in X^s \Longleftrightarrow (\psi_{+},\psi_{-})\in X_{+}^s\times X_{-}^s \text{ and }
\| \psi\|_{X^s}:= \| \psi_+\|_{X_+^s} +\| \psi\|_{X_-^s}.
\end{equation*}

It suffices to consider positive times. We will represent a solution of \eqref{system} using the Duhamel's formula on $[0,\infty]$
\begin{align}\label{duhamel}
\begin{aligned}
\psi_{+}(t)&=e^{-it\la D\ra_m}\Pi_{+}^m(D)\psi_0 + \sum_{\theta_1,\theta_2,\theta_3\in\{\pm\}}i\int_0^t e^{-i(t-t')\la D\ra_m}\Pi_{+}^m(D) \bigg( V*\la \psi_{\theta_1},\beta\psi_{\theta_2}\ra \beta\psi_{\theta_3} \bigg) (t') dt'\\
\psi_{-}(t)&=e^{it\la D\ra_m}\Pi_{-}^m(D)\psi_0 +\sum_{\theta_1,\theta_2,\theta_3\in\{\pm\}}i\int_0^t e^{i(t-t')\la D\ra_m}\Pi_{-}^m(D) \bigg( V*\la \psi_{\theta_1},\beta\psi_{\theta_2}\ra \beta\psi_{\theta_3} \bigg) (t') dt',
\end{aligned}
\end{align}
whenever the initial data satisfy $\|\psi_0\|_{H^s}\le \delta$.

For all $\psi_0\in H^s(\R^3)$ we immediately have
\begin{equation}\label{eq:lin}
\|\mathbf{1}_{[0,\infty)}e^{- it\la D\ra_m}\Pi_{+}^m(D)\psi_0\|_{X_{+}^s} 
+\|\mathbf{1}_{[0,\infty)}e^{ it\la D\ra_m}\Pi_{-}^m(D)\psi_0\|_{X_{-}^s}
\approx \|\psi_0\|_{H^s}.
\end{equation}
Next we study the nonlinear part. We consider more general situation: For $s>0$ and all $\psi_{i}=\Pi_{\theta_i}^m(D)\psi_{i}$ for $i=1,2,3$ we claim that
\begin{equation}\label{eq:nonlin-est}
\Big\| 
\int_0^t e^{\mp i(t-t')\la D\ra_m} \Pi_{\pm}^m(D) \Big( V*\la \psi_{1},\beta\psi_{2}\ra \beta\psi_{3} \Big) (t') dt'  \Big\|_{X_{\pm}^s}
\lesssim
\prod_{i=1}^3\|\psi_{i}\|_{X^s}.
\end{equation}

For a moment we assume \eqref{eq:nonlin-est} and prove the Theorem~\ref{thm:scattering}.
Let $T(\psi)$ denote the nonlinear operator defined by the sum of the right hand sides of \eqref{duhamel}. Then from \eqref{eq:lin} and \eqref{eq:nonlin-est} we conclude that
\begin{align}
\| T(\psi)\|_{X^s} &\ls \delta + \|\psi\|_{X^s}^3, \label{contraction}\\
\| T(\psi) - T(\phi) \| &\ls ( \|\psi\|_{X^s} + \|\phi\|_{X^s} )^2 \|\psi-\phi\|_{X^s}.
\end{align}
Indeed, the second one follows by taking $\psi_1=\psi-\phi,\psi_2=\psi,\psi_3=\phi$ or $\psi_1=\psi,\psi_2=\psi-\phi,\psi_3=\phi$ or
$\psi_1=\psi,\psi_2=\phi,\psi_3=\psi-\phi$ in \eqref{eq:nonlin-est}. 
Theorem~\ref{thm:scattering} now follows from the standard approach via the contraction mapping principle. In particular, the scattering claim follows from the fact
that functions in $V^2_{\mathbf{I}}$ have a limit at $\infty$. We omit
the details.

\subsection{Proof of \eqref{eq:nonlin-est}}
By duality and dyadic decomposition, square of the left side norm is bounded by 
(See \cite{HT-15})
\begin{align*}
&\sum_{k_4\in \Z} 2^{2sk_4}
\sup_{\|\psi_{k_4} \|_{V_{\pm\la D\ra}^{2}}=1}
\Big|
\sum_{k_i\in\Z }
\iint V*\la \psi_{1,k_1},\psi_{2,k_2}\ra_{\C^4}\la \psi_{3,k_3},\Pi_{\pm}^m(D)P_{k_4}^m\psi_{k_4}\ra_{\C^4} dtdx   
\Big|^2 \\
&\le \sum_{k_4\in \Z} 2^{2sk_4}
\sup_{\|\psi_{k_4} \|_{V_{\pm\la D\ra}^{2}}=1}
\Big(
\sum_{k, k_i\in\Z} 2^{-\gamma k}\la 2^k\ra^{\gamma-2}
\|P_k \la \psi_{1,k_1},\psi_{2,k_2}\ra \|_{L^2}
\|\widetilde{P_k} \la \psi_{3,k_3},\Pi_{\pm}^m(D)\psi_{4,k_4}\ra \|_{L^2} 
\Big)^2 \\
&\le I_1+I_2+I_3, 
\end{align*}
where $I_1:=\sum_{ 2^{k}\sim 2^{k_3}}$,  $I_2:=\sum_{ 2^{k} \gg 2^{k_3}}$  and $I_3:=\sum_{2^{k}\ll 2^{k_3}}$.

All the estimates below will be obtained in $V^2$ space where $U^2$ space is continuously embedded. So we will actually show the stronger ones.

We suffices to show \eqref{eq:nonlin-est} for $0<s<\min(\frac12,\frac{2-\gamma}{2})$.
Once we make this assumption we have for all $k\in\Z$
$$2^{(2-\gamma-2s) k}\la 2^{k}\ra^{\gamma-2} \le 1,$$
which will be repeatedly used in the following proof. 

\medskip\medskip\medskip

{\it (1)Estimates for $I_1$:}
We further divide the case $I_1\le I_{11} +I_{12} +I_{13}$ where $I_{11}:=\sum_{2^{k_1}\sim 2^{k_2}}$, $I_{12}:=\sum_{2^{k_1}\gg 2^{k_2}}$ and $I_{13}:=\sum_{2^{k_1}\ll 2^{k_2}}$.

We estimate $I_{11}$ by applying \eqref{bilinear:hhl:high}
\begin{align*}
I_{11}&\ls  \sum_{k_4\in \Z} 2^{2sk_4}
\Big(
\sum_{\substack{k\in \Z, 2^{k}\ls 2^{k_1}\sim 2^{k_2} \\ 2^{k_3}\sim 2^{k_4}}} 2^{-\gamma k}\la 2^k\ra^{\gamma-2} B_m(k,k_1,k_2)B_m(k,k_3,k_4)
\| \psi_{1,k_1} \|_{V_{\theta_1\la D\ra}^2} \| \psi_{2,k_2} \|_{V_{\theta_2\la D\ra}^2}
\| \psi_{3,k_3}\|_{V_{\theta_3\la D\ra}^2}
\Big)^2  \\
&\ls\sum_{k_3\in \Z} 2^{2sk_3} \| \psi_{3,k_3}\|_{V_{\theta_3\la D\ra}^2}^2
\Big(
\sum_{k\in \Z, 2^k\ls 2^{k_1}} \frac{2^{-\gamma k}\la 2^k\ra^{\gamma-2} B_m(k,k_1,k_2)B_m(k,k_3,k_4)}{2^{2sk_1}}
2^{sk_1}\| \psi_{1,k_1} \|_{V_{\theta_1\la D\ra}^2} 2^{sk_1}\| \psi_{2,k_1} \|_{V_{\theta_2\la D\ra}^2}
\Big)^2.
\end{align*}
For $m=0$, we have
\begin{align*}
I_{11} &\ls \| \psi_3 \|_{X_{\theta_3}^s}^2 \Big(
\sum_{ 2^k\ls 2^{k_1}} 2^{(2-\gamma -2s)k}\la 2^k\ra^{\gamma-2}  2^{2s(k-k_1)}
2^{sk_1}\| \psi_{1,k_1} \|_{V_{\theta_1|D|}^2} 2^{sk_1}\| \psi_{2,k_1} \|_{V_{\theta_2|D|}^2}
\Big)^2 \\
&\ls \| \psi_3 \|_{X_{\theta_3}^s}^2 \Big(
\sum_{ 2^k\ls 2^{k_1}}  2^{2s(k-k_1)}
2^{sk_1}\| \psi_{1,k_1} \|_{V_{\theta_1|D|}^2} 2^{sk_1}\| \psi_{2,k_1} \|_{V_{\theta_2|D|}^2}
\Big)^2 \\
&\ls \| \psi_1 \|_{X_{\theta_1}^s}^2\| \psi_2 \|_{X_{\theta_2}^s}^2\| \psi_3 \|_{X_{\theta_3}^s}^2,
\end{align*}
where we used $0\le\gamma<2$, $0<s<\frac{2-\gamma}{2}$ and Cauchy Schwarz inequality.

For $m>0$ we further divide $I_{11}\ls I_{11}^1+ I_{11}^2 +I_{11}^3+I_{11}^4$, where 
$I_{11}^1$ is summation over $\min(k_1,k_2)>1\ \& \ \min(k_3,k_4)>1$,
$I_{11}^2$ over $\min(k_1,k_2)=1\ \& \ \min(k_3,k_4)>1$,
$I_{11}^3$ over $\min(k_1,k_2)>1\ \& \ \min(k_3,k_4)=1$, and
$I_{11}^4$ over $\min(k_1,k_2)=\min(k_3,k_4)=1$.
The case $I_{11}^1$, since $B_m(k,k_1,k_2)B_m(k,k_3,k_4)=2^{2k}$, can be bounded by the same method as $m=0$.
For $I_{11}^4$, we have 
$$B_m(k,k_1,k_2)B_m(k,k_3,k_4)=2^{k}.$$ And the worst term occurs when all $k_i$ is zero
\begin{align*} 
& 2^{2s} \| \psi_{3,\le0}\|_{V_{\theta_3\la D\ra_m}^2}^2
\Big(
\sum_{2^k\ls 0} \frac{2^{(1-\gamma) k}} {2^{2s}}
2^{s}\| \psi_{1,\le 0} \|_{V_{\theta_1\la D\ra_m}^2} 2^{s}\| \psi_{2,\le 0} \|_{V_{\theta_2\la D\ra_m}^2}
\Big)^2 \\
&\qquad \ls  2^{2s} \| \psi_{1,\le0}\|_{V_{\theta_1\la D\ra_m}^2}^2
 2^{2s} \| \psi_{2,\le0}\|_{V_{\theta_2\la D\ra_m}^2}^2
  2^{2s} \| \psi_{3,\le0}\|_{V_{\theta_3\la D\ra_m}^2}^2,
\end{align*}
where the assumption $0\le\gamma<1$ is essential to make summation over $k$ finite. Estimates for $I_{11}^2$ and $I_{11}^3$ can be done similarly.

Next we estimate $I_{12}$. We can exclude the case where $m>0$ and $2^{k_3}\sim 1$. Indeed, if $2^{k_3}\sim 1$, there hold $2^{k_2}\ll 0$ which means $P_{k_2}^m=0$, or $2^{k_2}\sim 2^{k_1}$ which has already been considered in $I_{11}$. Then we apply \eqref{bilinear:HighLow} and \eqref{bilinear:hhl:high} 
\begin{align*}
I_{12} &\ls    \sum_{k_4\in \Z} 2^{2sk_4}
\Big(
\sum_{\substack{k\in \Z, 2^k\sim 2^{k_1}\gg 2^{k_2} \\ 2^{k_3}\sim 2^{k_4}}} 2^{-\gamma k_1}\la 2^{k_1}\ra^{\gamma-2} 2^{sk_1} 2^{(1-s)k_2} 
\| \psi_{1,k_1} \|_{V_{\theta_1\la D\ra}^2} \| \psi_{2,k_2} \|_{V_{\theta_2\la D\ra}^2}
2^k \| \psi_{3,k_3}\|_{V_{\theta_3\la D\ra}^2}
\Big)^2  \\
&\ls    \sum_{k_3\in \Z} 2^{2sk_3} \| \psi_{3,k_3}\|_{V_{\theta_3\la D\ra}^2}^2
\Big(
\sum_{ 2^{k_1}\gg 2^{k_2}} 2^{(2-\gamma-2s) k_1}\la 2^{k_1}\ra^{\gamma-2} 2^{(1-2s)(k_2-k_1)}
2^{sk_1}\| \psi_{1,k_1} \|_{V_{\theta_1\la D\ra}^2} 
2^{sk_2}\| \psi_{2,k_2} \|_{V_{\theta_2\la D\ra}^2}
\Big)^2  \\
&\ls \| \psi_1 \|_{X_{\theta_1}^s}^2\| \psi_2 \|_{X_{\theta_2}^s}^2\| \psi_3 \|_{X_{\theta_3}^s}^2,
\end{align*}
since $2^{(2-\gamma-2s) k_1}\la 2^{k_1}\ra^{\gamma-2}\le1$ and $0<s<\frac12$.

$I_{13}$ is similarly bounded as $I_{12}$ by just exchanging the role of $\psi_1$ and $\psi_2$.

\medskip\medskip

{\it (2)Estimates for $I_2$:} 
We divide the case $I_2\le I_{21} +I_{22} +I_{23}$ where $I_{21}:=\sum_{2^{k_1}\sim 2^{k_2}}$, $I_{22}:=\sum_{2^{k_1}\gg 2^{k_2}}$ and $I_{23}:=\sum_{2^{k_1}\ll 2^{k_2}}$.

Consider $I_{21}$. In this range it holds
$2^{k_1}\sim 2^{k_2} \gtrsim 2^{k_4}\sim 2^{k} \gg 2^{k_3}$.
We may assume $2^{k_1},2^{k_2}>1$ if $m>0$. Then by \eqref{bilinear:hhl:high} and \eqref{bilinear:HighLow} we obtain
\begin{align*}
I_{21}&\ls  \sum_{k_4\in \Z} 2^{2sk_4}
\Big(
\sum_{\substack{ 2^{k_1}\sim 2^{k_2} \\ k\in \Z, 2^k\sim 2^{k_4}\gg 2^{k_3}}} 2^{-\gamma k}\la 2^k\ra^{\gamma-2} 2^k
\| \psi_{1,k_1} \|_{V_{\theta_1\la D\ra}^2} \| \psi_{2,k_2} \|_{V_{\theta_2\la D\ra}^2}
2^{sk_4}2^{(1-s)k_3}\| \psi_{3,k_3}\|_{V_{\theta_3\la D\ra}^2}
\Big)^2  \\
&\ls \sum_{k_4\in \Z} 
\Big(\sum_{\substack{ 2^{k_1}\sim 2^{k_2} \\ k\in \Z, 2^k\sim 2^{k_4}\gg 2^{k_3}}} 2^{(1-\gamma) k}\la 2^k\ra^{\gamma-2} 
2^{sk_1}\| \psi_{1,k_1} \|_{V_{\theta_1\la D\ra}^2} 
2^{sk_2} \| \psi_{2,k_2} \|_{V_{\theta_2\la D\ra}^2}
2^{(1-s)k_3}\| \psi_{3,k_3}\|_{V_{\theta_3\la D\ra}^2} 
\Big)^2 \\
&\ls \| \psi_1 \|_{X_{\theta_1}^s}^2\| \psi_2 \|_{X_{\theta_2}^s}^2
\sum_{k_4\in\Z} \Big( \sum_{2^{k_4}\gg 2^{k_3}} 2^{(2-\gamma-2s) k_4}\la 2^{k_4}\ra^{\gamma-2} 2^{(1-2s)(k_3-k_4)} 
2^{sk_3}\| \psi_{3,k_3}\|_{V_{\theta_3\la D\ra}^2} \Big)^2\\
&\ls \| \psi_1 \|_{X_{\theta_1}^s}^2\| \psi_2 \|_{X_{\theta_2}^s}^2\| \psi_3 \|_{X_{\theta_3}^s}^2.
\end{align*}

Next consider $I_{22}$. In this range we have $2^{k} \sim 2^{k_1} \sim 2^{k_4}$. 
We estimate
\begin{align*}
I_{22} &\ls  \sum_{k_4\in \Z} 2^{2sk_4}
\Big(
\sum_{\substack{ 2^{k_1}\gg 2^{k_2} \\ k\in \Z, 2^k\sim 2^{k_4}\gg 2^{k_3}}} 2^{-\gamma k}\la 2^k\ra^{\gamma-2}  2^{sk_1}2^{(1-s)k_2}
\| \psi_{1,k_1} \|_{V_{\theta_1\la D\ra}^2} \| \psi_{2,k_2} \|_{V_{\theta_2\la D\ra}^2}
2^{sk_4}2^{(1-s)k_3}\| \psi_{3,k_3}\|_{V_{\theta_3\la D\ra}^2}
\Big)^2  \\
&\ls  \sum_{k_4\in\Z} 
\Big(
\sum_{\substack{ 2^{k_4}\gg 2^{k_2} \\ 2^{k_4}\gg 2^{k_3}}} 2^{(-\gamma+2s)k_4} \la 2^{k_4}\ra^{\gamma-2} 2^{(1-2s)k_2}
2^{sk_4}\| \psi_{1,k_4} \|_{V_{\theta_1\la D\ra}^2} 
2^{sk_2}\| \psi_{2,k_2} \|_{V_{\theta_2\la D\ra}^2}
2^{(1-2s)k_3}
2^{sk_3}\| \psi_{3,k_3}\|_{V_{\theta_3\la D\ra}^2}
\Big)^2 \\
&\ls  \sum_{k_4\in\Z} 2^{2sk_4}\| \psi_{1,k_4} \|_{V_{\theta_1\la D\ra}^2}^2 
\Big(
\sum_{ 2^{k_4}\gg 2^{k_2}} 2^{(1-\frac\gamma2-s)k_4} \la 2^{k_4}\ra^{\frac\gamma2-1} 2^{(1-2s)(k_2-k_4)}
2^{sk_2}\| \psi_{2,k_2} \|_{V_{\theta_2\la D\ra}^2}\Big)^2 \\
&\qquad\qquad\times
\Big(\sum_{2^{k_4}\gg 2^{k_3}} 2^{(1-\frac\gamma2-s)k_4} \la 2^{k_4}\ra^{\frac\gamma2-1}2^{(1-2s)(k_3-k_4)}
2^{sk_3}\| \psi_{3,k_3}\|_{V_{\theta_3\la D\ra}^2}
\Big)^2 \\
&\ls \| \psi_1 \|_{X_{\theta_1}^s}^2\| \psi_2 \|_{X_{\theta_2}^s}^2\| \psi_3 \|_{X_{\theta_3}^s}^2.
\end{align*}

The estimates for $I_{23}$ can be shown by the same method as $I_{22}$.

\medskip\medskip
{\it(3) Estimates for $I_3$:} We divide the case $I_2\le I_{31} +I_{32} +I_{33}$ where $I_{31}:=\sum_{2^{k_1}\sim 2^{k_2}}$, $I_{32}:=\sum_{2^{k_1}\gg 2^{k_2}}$ and $I_{33}:=\sum_{2^{k_1}\ll 2^{k_2}}$.

Consider $I_{31}$. We may assume if $m>0$ then $2^{k_1},2^{k_2}>1$.
\begin{align*}
I_{31}&\ls  \sum_{k_4\in \Z} 
\Big(
\sum_{\substack{ 2^{k_1}\sim 2^{k_2} \\ k\in \Z, 2^k\sim 2^{k_3}\gg 2^{k_4}}} 2^{-\gamma k}\la 2^k\ra^{\gamma-2} 2^k
\| \psi_{1,k_1} \|_{V_{\theta_1\la D\ra}^2} \| \psi_{2,k_2} \|_{V_{\theta_2\la D\ra}^2}
2^{k_4}2^{sk_3}\| \psi_{3,k_3}\|_{V_{\theta_1\la D\ra}^2}
\Big)^2  \\
&\ls \sum_{k_4\in \Z} 
\Big(\sum_{ 2^{k_1} \gtrsim  2^{k_3}\gg 2^{k_4}} 2^{(1-\gamma) k_3}\la 2^{k_3}\ra^{\gamma-2}2^{-2sk_1} 2^{k_4}
2^{sk_1} \| \psi_{1,k_1} \|_{V_{\theta_1\la D\ra}^2} 
2^{sk_1} \| \psi_{2,k_1} \|_{V_{\theta_2\la D\ra}^2}
2^{sk_3} \| \psi_{3,k_3}\|_{V_{\theta_1\la D\ra}^2} 
\Big)^2 \\
&\ls \| \psi_1 \|_{X_{\theta_1}^s}^2\| \psi_2 \|_{X_{\theta_2}^s}^2 \sum_{k_4\in \Z} 
\Big(\sum_{ 2^{k_3}\gg 2^{k_4}} 2^{(2-\gamma-2s) k_3}\la 2^{k_3}\ra^{\gamma-2} 2^{k_4-k_3}
2^{sk_3} \| \psi_{3,k_3}\|_{V_{\theta_1\la D\ra}^2}
\Big)^2 \\
&\ls \| \psi_1 \|_{X_{\theta_1}^s}^2\| \psi_2 \|_{X_{\theta_2}^s}^2\| \psi_3 \|_{X_{\theta_3}^s}^2.
\end{align*}

Next consider $I_{32}$. In this range we have $2^k\sim2^{k_1}\sim2^{k_3}$. We estimate
\begin{align*}
I_{32} &\ls  \sum_{k_4\in \Z} 
\Big(
\sum_{\substack{ k\in \Z, 2^k \sim 2^{k_1}\gg 2^{k_2} \\2^{k_3}\gg 2^{k_4}}} 2^{-\gamma k}\la 2^k\ra^{\gamma-2}  2^{(1-2s)k_2}
2^{sk_1}\| \psi_{1,k_1} \|_{V_{\theta_1\la D\ra}^2} 
2^{sk_2}\| \psi_{2,k_2} \|_{V_{\theta_2\la D\ra}^2}
 2^{k_4}
 2^{sk_3}\|\psi_{3,k_3}\|_{V_{\theta_3\la D\ra}^2}
\Big)^2  \\
&\ls  \sum_{k_4\in\Z} 
\Big(
\sum_{\substack{ 2^{k_1}\gg 2^{k_2} \\ 2^{k_1}\gg 2^{k_4}}} 2^{-\gamma k_1}\la 2^{k_1}\ra^{\gamma-2}  2^{(1-2s)k_2}2^{k_4}
2^{sk_1}\| \psi_{1,k_1} \|_{V_{\theta_1\la D\ra}^2} 
2^{sk_2}\| \psi_{2,k_2} \|_{V_{\theta_2\la D\ra}^2}
2^{sk_1}\|\psi_{3,k_1}\|_{V_{\theta_3\la D\ra}^2}
\Big)^2 \\
&\ls \| \psi_2 \|_{X_{\theta_2}^s}^2 \sum_{k_4\in\Z} 
\Big(
\sum_{2^{k_1}\gg 2^{k_4}} 2^{(2-\gamma-2s) k_1}\la 2^{k_1}\ra^{\gamma-2}  2^{k_4-k_1}
2^{sk_1}\| \psi_{1,k_1} \|_{V_{\theta_1\la D\ra}^2} 
2^{sk_1}\|\psi_{3,k_1}\|_{V_{\theta_3\la D\ra}^2}
\Big)^2 \\
&\ls \| \psi_1 \|_{X_{\theta_1}^s}^2\| \psi_2 \|_{X_{\theta_2}^s}^2\| \psi_3 \|_{X_{\theta_3}^s}^2.
\end{align*}

Estimates for $I_{33}$ is similar.


\section{ill-posedness}
In this section we consider the supercritical range where the initial data is given in $H^s(\R^3)$ for $s<0$.
We make further assumption on potential $V$ that $\widehat V(\xi)$ is positive for $|\xi|\gg1$. The Coulomb and Yukawa potentials are still in our consideration. 
We provide the ill-posed result which shows the nonlinear term estimates on \eqref{contraction} essential to occur the contraction fail for any resolution space $\widetilde{X}^s$. We adapt the argument in \cite{MSetal-01}, where detailed explanation is well arranged. We suffices to show the following:
\begin{thm}
	\label{thm:Illposed}
	Let $m\ge0$. For fixed $T>0$ and $s<0$ 
	the inequality
	\begin{equation}\label{ineq:tri}
	\sup_{t\in [0,T]} 	\Big\| \int_{0}^{t} U_m(t-\tau,D) 
	\big(V* \la U_m(\tau,D)\psi , \beta U_m(\tau,D)\psi \ra \beta U_m(\tau,D)\psi \big)(\tau) d\tau
	\Big\|_{H^{s}(\R^3)}
	\lesssim
	\|\psi\|_{H^s(\R^3)}^3
	\end{equation}
	fails to hold for all $\psi\in H^{s}(\mathbb{R}^{3})$. 
\end{thm}

\begin{proof}
	Our proof is based on the modification of  \cite[Proposition3.1]{HL-14}. 
	For $\lambda\gg1$ we define the annulus
	\begin{align*}
	W_\lambda=\{ x\in\R^3 : \lambda \le |x| \le 2\lambda \}.
	\end{align*}
	Let $\phi:\R^3\rightarrow\C$ be the inverse fourier transform of the characteristic function	$\chi_{W_{\lambda}}$. And we choose $\psi=(\phi,0,0,0)$. 
	Obviously, $\|\psi\|_{H^s}\ls \lambda^{\frac32+s}$.
	Next, we consider
	\begin{equation}\label{def:Ft}
	N(t,\xi):=
	\mathcal{F}_x\bigg(
\int_{0}^{t} U_m(t-\tau,D) 
\big(V* \la U_m(\tau,D)\psi , \beta U_m(\tau,D)\psi \ra \beta U_m(\tau,D) \psi \big) d\tau
	\bigg).
	\end{equation}
Our aim is to prove that for $t=\epsilon\lambda^{-1}$ with $0<\epsilon\ll1$ and $\xi\in W_{\lambda}$,
\begin{equation}\label{lowbound}
|N(t,\xi)|
\gtrsim
|t|\lambda^{4}=\eps\lambda^3.
\end{equation}	
Assuming that $\eqref{lowbound}$ holds,
the claim follows since the validity of 
$\eqref{ineq:tri}$ implies 
\[
\epsilon\lambda^{s+\frac92} 
\lesssim
\|\langle\xi\rangle^{s}N(t,\xi)\|_{L^{2}(\mathbb{R}^{3})}
\lesssim
\lambda^{3s+\frac{9}{2}},
\]
which is equivalent to
$
\epsilon\lesssim
\lambda^{2s}
$ for fixed $\epsilon>0$.
And this can hold as $\lambda\rightarrow\infty$ only if $s\ge0$. 
Hence, it suffices to show $\eqref{lowbound}$.	

We compute the Fourier transform   
\begin{align*}
N(t,\xi)&\approx
\iint_{\R^3\times\R^3}\int_0^t  U_m(t-\tau,\xi) \widehat{V}(\eta) 
\la U_m(\tau,\eta-\sigma)\widehat\psi(\eta-\sigma),\beta U_m(-\tau,-\sigma)\widehat\psi(-\sigma)  \ra \\
&\qquad\qquad\qquad \times
\beta U_m(\tau,\xi-\eta) \widehat{\psi}(\xi-\eta)   d\tau d\sigma d\eta. 
\end{align*}
Let us denote $j$th component of $x\in\C^4$ by $x_j$ and the $(i,j)$ entry of $4\times4$ matrix $A$ by $A_{ij}$. Putting $\widehat{\psi}=(\chi_{W_{\lambda}},0,0,0)$ we compute
\begin{align*}
[N(t,\xi)]_j&\approx
\iint_{\R^3\times\R^3}\int_0^t \Big[ U_m(t-\tau,\xi) \widehat{V}(\eta) 
\la U_m(\tau,\eta-\sigma)\widehat\psi(\eta-\sigma),\beta U_m(-\tau,-\sigma)\widehat\psi(-\sigma)  \ra \\
&\qquad\qquad\qquad \times
\beta U_m(\tau,\xi-\eta) \widehat{\psi}(\xi-\eta) \Big]_j  d\tau d\sigma d\eta \\
&=\iint_{\R^3\times\R^3}\int_0^t \widehat{V}(\eta) [ U_m(\tau,\eta-\sigma)^{T}\beta U_m(-\tau,\sigma)]_{11}\chi_{W_{\lambda}}(\eta-\sigma)\chi_{W_{\lambda}}(\sigma) \\
&\qquad\qquad\qquad \times
[U_m(t-\tau,\xi)\beta U_m(\tau,\xi-\eta)]_{j1} \chi_{W_{\lambda}}(\xi-\eta) d\tau d\sigma d\eta \\
&=\iint_{\R^3\times\R^3}\widehat{V}(\eta) \int_0^t [ U_m(\tau,\eta-\sigma)^{T}\beta U_m(-\tau,\sigma)]_{11} 
[U_m(t-\tau,\xi)\beta U_m(\tau,\xi-\eta)]_{j1}d\tau  \\
&\qquad\qquad\qquad \times1 \chi_{W_{\lambda}}(\eta-\sigma)\chi_{W_{\lambda}}(\sigma)\chi_{W_{\lambda}}(\xi-\eta) d\sigma d\eta.
\end{align*}
Then we estimate
\begin{align*}
|N(t,\xi)|&\ge |[N(t,\xi)]_1| \\ 
&\ge \Big| \iint_{\R^3\times\R^3}\widehat{V}(\eta) \int_0^t \Real \Big(
[ U_m(\tau,\eta-\sigma)^{T}\beta U_m(-\tau,\sigma)]_{11} 
 [U_m(t-\tau,\xi)\beta U_m(\tau,\xi-\eta)]_{11} \Big)d\tau  \\
&\qquad\qquad\qquad \times \chi_{W_{\lambda}}(\eta-\sigma)\chi_{W_{\lambda}}(\sigma)\chi_{W_{\lambda}}(\xi-\eta) d\sigma d\eta \Big| .
\end{align*}
From Lemma~\ref{Lem:lowerbound} we find the integration over $\tau$ is bounded below for $\lambda\gg1$
\begin{align*}
\int_0^t \Real\Big( [ U_m(\tau,\eta-\sigma)^{T}\beta U_m(-\tau,\sigma)]_{11} [U_m(t-\tau,\xi)\beta U_m(\tau,\xi-\eta)]_{11} \Big) d\tau 
\gtrsim t.
\end{align*} 
Since $\widehat{V}$ is positive we finally obtain
\begin{align*}
|N(t,\xi)|  &\gtrsim |t| \int_{\R^3} \int_{\R^3} \widehat{V}(\eta) \chi_{W_{\lambda}}(\eta-\sigma)\chi_{W_{\lambda}}(\sigma)\chi_{W_{\lambda}}(\xi-\eta)
d\eta d\sigma  \\
&\gtrsim |t| \lambda^{4}.
\end{align*}
\end{proof}


\section{Appendix}
\begin{lem}\label{Lem:lowerbound}
Let $m\ge0$, $\xi,\eta\in\R^3$ and $t,\tau\in\R$. Suppose $|\xi|,|\eta|\sim \lambda$ and $|\tau|,|t|\le \epsilon\lambda^{-1}$ with $0<\epsilon\ll1$. Then we have for sufficiently large $\lambda\gg1$
\begin{align}\label{lowerbound}
\begin{aligned}
\Real[ U_m(\tau,\xi)\beta U_m(t,\eta)]_{11} &\gtrsim 1 \text{ and }
\Imag [ U_m(\tau,\xi)\beta U_m(t,\eta)]_{11} \lesssim \lambda^{-1},\\
\Real [ U_m(\tau,\xi)^{\intercal}\beta U_m(t,\eta)]_{11} &\gtrsim 1 \text{ and }
\Imag [ U_m(\tau,\xi)^{\intercal}\beta U_m(t,\eta)]_{11} \lesssim \lambda^{-1},
\end{aligned}\end{align} 
where the implicit constants depend only on $\epsilon$ and $m$.
\end{lem}
\begin{proof}
For $|\xi|\sim \lambda$ and $\tau\le \epsilon\lambda^{-1}$ with $0<\epsilon\ll1$ we have
\begin{equation*}
\cos\big(\tau\la \xi\ra_m\big) \gtrsim 1 \text{ and } |\la \xi\ra_m^{-1}\sin\big( \tau\la \xi\ra_m)| \le \epsilon \lambda^{-1}.
\end{equation*}	
And we directly compute from \eqref{Def:U} and \eqref{matrix}
\begin{align*}
&\Real [ U_m(\tau,\xi)\beta U_m(t,\eta)]_{11}
=\cos\big(\tau\la \xi\ra_m\big)\cos\big(t\la \eta\ra_m\big)
+2\la \xi\ra_m^{-1}\sin\big( \tau\la \xi\ra_m) \la\eta\ra_m^{-1}\sin\big( t\la \eta\ra_m) \\
&\Imag [ U_m(\tau,\xi)\beta U_m(t,\eta)]_{11} = -\cos\big(\tau\la \xi\ra_m\big)\la\eta\ra_m^{-1}\sin\big( t\la \eta\ra_m)
-\cos\big(t\la \eta\ra_m\big)\la \xi\ra_m^{-1}\sin\big( \tau\la \xi\ra_m) \\
&\Real [ U_m(\tau,\xi)^\intercal\beta U_m(t,\eta)]_{11}
=\cos\big(\tau\la \xi\ra_m\big)\cos\big(t\la \eta\ra_m\big) \\
&\Imag [ U_m(\tau,\xi)^\intercal\beta U_m(t,\eta)]_{11}
= 2\la \xi\ra_m^{-1}\sin\big( \tau\la \xi\ra_m) \la\eta\ra_m^{-1}\sin\big( t\la \eta\ra_m)
-\cos\big(\tau\la \xi\ra_m\big)\la\eta\ra_m^{-1}\sin\big( t\la \eta\ra_m) \\
&\qquad\qquad\qquad\qquad\qquad\qquad\quad
-\cos\big(t\la \eta\ra_m\big)\la \xi\ra_m^{-1}\sin\big( \tau\la \xi\ra_m).
\end{align*}
Thus we have \eqref{lowerbound}.
\end{proof}  

\medskip\medskip

\begin{proof}[Proof of Proposition~\ref{Prop;V2}]
We only consider $(+,+)$ case.

{\it Step1:} Let
$P=\sum_{k_3\in F} P_{k_3}$, with a finite set $F$ of
integer $k_3$ of size $2^{k_3}\sim 2^{k_1}\sim2^{k_2}$, such that
$P P_{k_i} =P_{k_i}$ for $i=1,2$. We claim first that
\begin{equation}\label{eq:u4}
\|P_{\leq k} \la P \psi, P \psi\ra_{\R^4} \|_{L^2} \ls   2^k \|\psi\|_{V^2_{\la D\ra_m}}^2
\end{equation}
for any $2^k\ls 2^{k_1}$ and real-valued $\psi\in V^2_{\la D\ra_m}$. To prove \eqref{eq:u4}, let
$\omega_k=(\widecheck{\rho}_{\leq k+1})^2$. Then, $\omega_k \geq 0$ and
we have the pointwise bound
\[\rho_{\leq k}\ls \rho_{\leq k+1} \ast \rho_{\leq k+1} \ls
\rho_{\leq k+2}\] on the Fourier side, which implies
\[ \|P_{\leq k} \la P\psi,P\psi\ra_{\R^4} \|_{L^2}
\ls \|\omega_k \ast \la P\psi,P\psi\ra_{\R^4}\|_{L^2}
\ls \|P_{\leq k+2} \la P\psi,P\psi\ra_{\R^4}\|_{L^2}.\]
So it suffices to show that
\begin{align}\label{claim}
\|\omega_k \ast \la P\psi,P\psi\ra_{\R^4}\|_{L^2} 
\ls  2^k \|\psi\|_{V^2_{\la D\ra_m}}^2.
\end{align}

For real valued $f$, the quantity
\[n(f):=\|T(f)\|_{L_x^4(\R^{3})}, \text{ for } T(f)=\Big( \omega_k
\ast \la P f, P f\ra_{\R^4} \Big)^{\frac12},\]
is subadditive. 
Indeed, since $\la f,f\ra_{\R^4}=|f|^2$ by Cauchy-Schwarz inequality  we estimate
\begin{align*}
T^2(f+g)(x)={}&\int_{\R^3} \omega_k (x-y) \la P f(y)+ Pg(y) , P f(y)+ Pg(y) \ra_{\R^4} dy\\
={}&\int_{\R^3} \omega_k (x-y) \la  P f(y)+ Pg(y), Pf(y)\ra_{\R^4}  dy\\&{}+\int_{\R^3} \omega_k (x-y) \la  P f(y)+ Pg(y), Pg(y)\ra_{\R^4}  dy\\
\leq{}& T(f+g)(x)T(f)(x)+T(f+g)(x)T(g)(x),
\end{align*}
which implies $T(f+g)\leq T(f)+T(g)$. From this it follows that
\[
n(f+g)\leq \|T(f)+T(g)\|_{L_x^4}\leq n(f)+n(g).
\]
Also, we obviously have $n(c f)=|c|n(f)$ for all $c\in\C$.
Due to \eqref{Bilinear estimate} we have
\begin{equation}\label{eq:n-est}
\|n( e^{-it\la D\ra_m}f)\|_{L^4_t}
\ls \| \omega_k\ast \la e^{-it\la D\ra_m}f, e^{-it\la D\ra_m}f\ra_{\R^4} \|_{L_t^2L_x^2}^\frac12
\ls 2^{\frac k2} \|f\|_{L^2}
\end{equation}
for all  $f \in L^2(\R^3)$.  

Let
$\psi \in U^4_{\la D\ra_m}$ be with atomic decomposition
\[
\psi =\sum_{j} c_j a_j, \; \text{s.th. }\sum_{j} |c_j| \leq 2
\|\psi\|_{U^4_{\la D\ra_m}}, \text{ and }
U^4_{\la D\ra_m}\text{-atoms } a_j.
\]
We have
\begin{equation}\label{eq:n-est-u4}
\|n(\psi )\|_{L^4_t}\leq \sum_{j} |c_j| \|n(a_j)\|_{L^4_t} \ls
2^{\frac k2}\|\psi\|_{U^4_{\la D\ra_m}},
\end{equation}
provided that for any $U^4_{\la D\ra_m}$-atom $a$ the estimate
\[
\|n(a)\|_{L^4_t}\ls 2^{\frac k2}
\]
holds true. Indeed, let
$a(t)=\sum_k\mathbf{1}_{I_k}(t) e^{-it\la D\ra_m}\varphi_k$, for some partition
$(I_k)$ of $\R$ and $\varphi_k\in L^2(\R^3)$ satisfying
$\sum_{k}\|\varphi_k\|_{L^2}^4\leq 1$. Then,
\begin{align*}
\|n(a)\|_{L^4_t}\leq{}& \Big\|\sum_k\mathbf{1}_{I_k}(t)
n(e^{-it\la D\ra_m}\varphi_k)\Big\|_{L^4_t}
\leq{} \Big(\sum_k  \|n(e^{-it\la D\ra_m}\varphi_k)\|_{L^4_t}^4\Big)^{\frac14}\\
\ls{}& 2^{\frac k2}\Big(\sum_k
\|\varphi_k\|_{L^2}^4\Big)^{\frac14}
\ls{}2^{\frac k2},
\end{align*}
where we used \eqref{eq:n-est} in the third inequality, which
completes the proof of \eqref{eq:n-est-u4}. This implies
\begin{equation*}
\|\omega_k\ast \la P\psi,P\psi\ra_{\R^4} \|_{L^2(\R^{1+3})}
= \|n(\psi(t))\|_{L^4_t}^2\ls
2^k \| \psi \|_{U^4_{\la D\ra_m}}^2
\ls 2^k \| \psi \|_{V^2_{\la D\ra_m}}^2,
\end{equation*}
where we used $V^2_{\la D\ra_m}\hookrightarrow U^4_{\la D\ra_m}$. Hence the claim \eqref{eq:u4} is established.

{\it Step2:} Let $\phi_{j}:=P_{\lambda_j} \psi_j$, $j=1,2$. We may
assume $\|\phi_j\|_{V^2_{\mathbf{S}}}=1$. The functions
$\phi_\pm=\phi_1\pm \phi_2$ satisfy $\phi_\pm=P \phi_\pm$,
$\|\phi_\pm\|_{V^2 _{\mathbf{S}}}\ls 1$,
\begin{align*}
&\Real(\la \phi_1,\phi_2 \ra_{\R^4})
=\frac12 \Big( \la \Real\phi_+,\Real\phi_+ \ra_{\R^4} -\la \Real\phi_-,\Real\phi_- \ra_{\R^4} +\la \Imag\phi_+,\Imag\phi_+ \ra_{\R^4}-\la \Imag\phi_-,\Imag\phi_- \ra_{\R^4}\Big), \\
&\Imag(\la \phi_1,\phi_2 \ra_{\R^4})=\Real(-i\la \phi_1,\phi_2 \ra_{\R^4}).
\end{align*}
We have
\begin{align}\label{11}
\|P_{\leq k} \la \phi_1,\phi_2 \ra_{\R^4} \|_{L^2(\R^{1+3})}
\ls \|P_{\leq k} \Real \la \phi_1,\phi_2 \ra_{\R^4} \|_{L^2(\R^{1+3})}
+\|P_{\leq k} \Real \la -i\phi_1,\phi_2 \ra_{\R^4} \|_{L^2(\R^{1+3})}.
\end{align}
Since $P\Real\phi_{\pm}=\phi_{\pm}$ and $P\Imag\phi_{\pm}=\phi_{\pm}$, the  estimate \eqref{eq:u4} yields
\begin{align*}
\|P_{\leq k} \Real \la \phi_1,\phi_2 \ra_{\R^4} \|_{L^2(\R^{1+3})}
\ls{}& \|P_{\leq k} \la \Real\phi_+,\Real\phi_+ \ra_{\R^4} \|_{L^2(\R^{1+3})}
+ \|P_{\leq k} \la \Real\phi_-,\Real\phi_- \ra_{\R^4} \|_{L^2(\R^{1+3})} \\
&+\|P_{\leq k} \la \Imag\phi_+,\Imag\phi_+ \ra_{\R^4} \|_{L^2(\R^{1+3})}
+\|P_{\leq k} \la \Imag\phi_-,\Imag\phi_- \ra_{\R^4} \|_{L^2(\R^{1+3})} \\
\ls{}& 2^k \big( \|\Real \phi_+\|_{V^2_{\la D\ra_m}}^2
+\|\Real \phi_-\|_{V^2_{\la D\ra_m}}^2
+\|\Imag \phi_+\|_{V^2_{\la D\ra_m}}^2
+\|\Imag \phi_-\|_{V^2_{\la D\ra_m}}^2 \big)\\
\ls{}& 2^k\big( \|\phi_+\|_{V^2_{\la D\ra_m}}^2 +\|\phi_-\|_{V^2_{\la D\ra_m}}^2 \big) 
\ls 2^k.
\end{align*}
We can similarly bound the second term in \eqref{11} once we set $\widetilde{\phi}_{\pm}=-i\phi_1\pm\phi_2$.
Thus we finally obtain
\[
\|P_{\leq k} \la P_{k_1}^m \psi_1, P_{k_2}^m \psi_2\ra_{\R^4}  \|_{L^2(\R^{1+3})}
\ls 2^k
\|P_{k_1}^m \psi_1\|_{V^2_{\la D\ra_m}}\|P_{k_2}^m \psi_2\|_{V^2_{\la D\ra_m}},
\]
which completes the proof since $P_k= P_{\le k+1}-P_{\le k}$.
\end{proof}

\bibliographystyle{amsplain} \bibliography{DiracHartree}

\providecommand{\bysame}{\leavevmode\hbox to3em{\hrulefill}\thinspace}
\providecommand{\MR}{\relax\ifhmode\unskip\space\fi MR }
\providecommand{\MRhref}[2]{%
  \href{http://www.ams.org/mathscinet-getitem?mr=#1}{#2}
}
\providecommand{\href}[2]{#2}
\begin{thebibliography}{10}

\bibitem{BH-2015}
Ioan Bejenaru and Sebastian Herr, \emph{The cubic {D}irac equation: small
  initial data in {$H^1(\Bbb{R}^3)$}}, Comm. Math. Phys. \textbf{335} (2015),
  no.~1, 43--82. \MR{3314499}

\bibitem{BH-2017}
\bysame, \emph{On global well-posedness and scattering for the massive
  {D}irac-{K}lein-{G}ordon system}, J. Eur. Math. Soc. (JEMS) \textbf{19}
  (2017), no.~8, 2445--2467. \MR{3668064}

\bibitem{CG-76}
J.~M. Chadam and R.~T. Glassey, \emph{On the {M}axwell-{D}irac equations with
  zero magnetic field and their solution in two space dimensions}, J. Math.
  Anal. Appl. \textbf{53} (1976), no.~3, 495--507. \MR{0413833}

\bibitem{CT-2006}
Yonggeun Cho and Tohru Ozawa, \emph{On the semirelativistic {H}artree-type
  equation}, SIAM J. Math. Anal. \textbf{38} (2006), no.~4, 1060--1074.
  \MR{2274474}

\bibitem{COetal-11}
Yonggeun Cho, Tohru Ozawa, and Suxia Xia, \emph{Remarks on some dispersive
  estimates}, Commun. Pure Appl. Anal. \textbf{10} (2011), no.~4, 1121--1128.
  \MR{2787438}

\bibitem{AFS-2007}
Piero D'Ancona, Damiano Foschi, and Sigmund Selberg, \emph{Null structure and
  almost optimal local regularity for the {D}irac-{K}lein-{G}ordon system}, J.
  Eur. Math. Soc. (JEMS) \textbf{9} (2007), no.~4, 877--899. \MR{2341835}

\bibitem{DF1989}
Jo\~ao-Paulo Dias and M\'ario Figueira, \emph{On the existence of weak
  solutions for a nonlinear time dependent {D}irac equation}, Proc. Roy. Soc.
  Edinburgh Sect. A \textbf{113} (1989), no.~1-2, 149--158. \MR{1025460}

\bibitem{HHetal-09}
Martin Hadac, Sebastian Herr, and Herbert Koch, \emph{Well-posedness and
  scattering for the {KP}-{II} equation in a critical space}, Ann. Inst. H.
  Poincar\'e Anal. Non Lin\'eaire \textbf{26} (2009), no.~3, 917--941.
  \MR{2526409}

\bibitem{HHetal-10}
\bysame, \emph{Erratum to ``{W}ell-posedness and scattering for the {KP}-{II}
  equation in a critical space'' [{A}nn. {I}. {H}. {P}oincar\'e---{AN} 26 (3)
  (2009) 917--941]}, Ann. Inst. H. Poincar\'e Anal. Non Lin\'eaire \textbf{27}
  (2010), no.~3, 971--972. \MR{2629889}

\bibitem{HY-2017}
S.~{Herr} and C.~{Yang}, \emph{{Critical well-posedness and scattering results
  for fractional Hartree-type equations}}, to appear in Differential Integral
  Equations.

\bibitem{HL-14}
Sebastian Herr and Enno Lenzmann, \emph{The {B}oson star equation with initial
  data of low regularity}, Nonlinear Anal. \textbf{97} (2014), 125--137.
  \MR{3146377}

\bibitem{HT-15}
Sebastian Herr and Achenef Tesfahun, \emph{Small data scattering for
  semi-relativistic equations with {H}artree type nonlinearity}, J.
  Differential Equations \textbf{259} (2015), no.~10, 5510--5532. \MR{3377534}

\bibitem{KTetal-14}
Herbert {Koch}, Daniel {Tataru}, and Monica {Vi\c{s}an}, \emph{{Dispersive
  equations and nonlinear waves. Generalized Korteweg-de Vries, nonlinear
  Schr\"odinger, wave and Schr\"odinger maps.}}, Basel: Birkh\"auser/Springer,
  2014 (English).

\bibitem{MT2009}
Shuji Machihara and Kimitoshi Tsutaya, \emph{Scattering theory for the {D}irac
  equation with a non-local term}, Proc. Roy. Soc. Edinburgh Sect. A
  \textbf{139} (2009), no.~4, 867--878. \MR{2520560}

\bibitem{MSetal-01}
L.~Molinet, J.~C. Saut, and N.~Tzvetkov, \emph{Ill-posedness issues for the
  {B}enjamin-{O}no and related equations}, SIAM J. Math. Anal. \textbf{33}
  (2001), no.~4, 982--988 (electronic). \MR{1885293}

\bibitem{NT-12}
Makoto Nakamura and Kimitoshi Tsutaya, \emph{Scattering theory for the {D}irac
  equation of {H}artree type and the semirelativistic {H}artree equation},
  Nonlinear Anal. \textbf{75} (2012), no.~8, 3531--3542. \MR{2901335}

\bibitem{P-14}
Fabio Pusateri, \emph{Modified scattering for the boson star equation}, Commun.
  Math. Phys. \textbf{332} (2014), no.~3, 1203--1234. \MR{3262624}

\bibitem{Steinbook}
Elias~M. Stein, \emph{Harmonic analysis: real-variable methods, orthogonality,
  and oscillatory integrals}, Princeton Mathematical Series, vol.~43, Princeton
  University Press, Princeton, NJ, 1993, With the assistance of Timothy S.
  Murphy, Monographs in Harmonic Analysis, III. \MR{1232192}

\bibitem{T-2018}
A.~{Tesfahun}, \emph{{Small data scattering for cubic Dirac equation with
  Hartree type nonlinearity in $\R^{1+3}$}}, ArXiv e-prints (2017).

\end{thebibliography}

\end{document}